\documentclass[12pt,reqno]{amsart}

\usepackage[T1]{fontenc}
\usepackage[utf8]{inputenc}
\usepackage{amsmath,amssymb,ifthen}
\usepackage{fullpage}
\usepackage{graphicx,psfrag,subfigure}
\usepackage{color}
\usepackage{moreverb}
\usepackage{amsthm}
\usepackage{mathrsfs}
\usepackage{esint}  
\usepackage{graphicx}
\usepackage{psfrag}
\usepackage{hhline}
\usepackage{enumerate}
\usepackage{bbm}
\usepackage{tikz}
\usepackage{tikz-3dplot}
\usepackage{hyperref}

\def\bisect{{\tt bisect}}
\def\k{\underline k}
\def\uni#1{{\rm uni}(#1)}
\def\ex{{\rm ext}}
\def\act{{\rm act}}
\def\coarse{\bullet}
\def\fine{\circ}
\def\qson{q_{\rm son}}
\def\kproj{k_{\proj}}
\def\kloc{k_{\rm loc}}
\def\kapp{k_{\rm app}}

\def\Crel{C_{\rm rel}}
\def\Clin{C_{\rm lin}}
\def\qlin{q_{\rm lin}}
\def\copt{c_{\rm opt}}
\def\Copt{C_{\rm opt}}
\def\Cmin{C_{\rm min}}

\def\Cref{C_{\rm ref}}
\def\opt{{\rm opt}}

\def\Ceff{C_{\rm eff}}%
\def\Crel{C_{\rm rel}}%
\def\osc{{\rm osc}}%

\def\MM{\mathcal M}

\def\UU{\mathcal{U}}
\def\T{\mathbb{T}}
\def\R{{\mathbb R}}

\def\N{{\mathbb N}}

\def\KK{{\mathcal K}}
\def\NN{{\mathcal N}}

\def\OO{{\mathcal O}}

\def\TT{{\mathcal T}}
\def\XX{{\mathcal X}}

\def\YY{{\mathcal Y}}
\def\diam{{\rm diam}}
\def\C#1{C_{\rm #1}}
\def\edual#1#2{\langle#1\,,\,#2\rangle_{\LL}}

\def\norm#1#2{\|#1\|_{#2}}

\def\set#1#2{\big\{#1\,:\,#2\big\}}

\def\dual#1#2{\langle#1\,,\,#2\rangle}

\def\level{{\rm level}}

\def\LL{\mathcal{L}}
\def\div{{\rm div}}

\def\AA{\textit{\textbf{A}}}
\def\ff{\boldsymbol f}

\def\bb{\textit{\textbf b}}

\def\refine{{\tt ref}}

\def\refine{{\tt refine}}
\def\MM{\mathcal M}
\def\II{\mathcal I}

\def\supp{{\rm supp}}

\def\proj{\mathrm{proj}}

\def\Ctrace{C_{\rm trace}}

\def\Cpatch{C_{\rm patch}}

\def\Cell{C_{\rm ell}}
\def\Cclos{C_{\rm clos}}

\def\Cson{C_{\rm son}}

\def\Cinv{C_{\rm inv}}

\newcounter{constantsnumber}
\def\namec#1#2{%
  \ifthenelse{\equal{#1}{lipschitz}}{C_{\rm lip}}{%
  \ifthenelse{\equal{#1}{c:unifEquivLevel}}{C_{\rm level}}{%
  \ifthenelse{\equal{#1}{mark}}{C_{\rm mark}}{%
  \ifthenelse{\equal{#1}{basis}}{C_{\rm basis}}{%
  \ifthenelse{\equal{#1}{monotone}}{C_{\rm mon}}{%
  \ifthenelse{\equal{#1}{cea}}{C_{\mbox{\rm\scriptsize C\'ea}}}{%
  \ifthenelse{\equal{#1}{norm}}{C_{\rm norm}}{%
  \ifthenelse{\equal{#1}{mon}}{{C}_{\rm mon}}{
  \ifthenelse{\equal{#1}{lip}}{{C}_{\rm lip}}{
  \ifthenelse{\equal{#1}{monA}}{c_{\rm mon}}{
  \ifthenelse{\equal{#1}{lipA}}{c_{\rm lip}}{
  \ifthenelse{\equal{#1}{normequiv1}}{c_{\rm norm}}{ 
  \ifthenelse{\equal{#1}{inv}}{C_{\rm inv}}{ 
  \ifthenelse{\equal{#1}{inv2}}{\widetilde{C}_{\rm inv}}{ 
  \ifthenelse{\equal{#2}{newcounter}}{\refstepcounter{constantsnumber}\label{const#1}}{}C_{\ref{const#1}}}%
  }}}}}}}}}}}}}}

\numberwithin{equation}{section}
\numberwithin{figure}{section}
\newtheorem{theorem}{Theorem}[section]

\newtheorem{lemma}[theorem]{Lemma}

\newtheorem{algorithm}[theorem]{Algorithm}

\newtheorem{remark}[theorem]{Remark}

\def\revision#1{{\color{black}{#1}}}
\def\new#1{{\color{black}{#1}}}

\usepackage{fancyhdr}
\advance\footskip0.4cm
\advance\textheight-0.4cm
\calclayout
\newcommand*\patchAmsMathEnvironmentForLineno[1]{%
  \expandafter\let\csname old#1\expandafter\endcsname\csname #1\endcsname
  \expandafter\let\csname oldend#1\expandafter\endcsname\csname end#1\endcsname
  \renewenvironment{#1}%
     {\linenomath\csname old#1\endcsname}%
     {\csname oldend#1\endcsname\endlinenomath}}%
\newcommand*\patchBothAmsMathEnvironmentsForLineno[1]{%
  \patchAmsMathEnvironmentForLineno{#1}%
  \patchAmsMathEnvironmentForLineno{#1*}}%
\AtBeginDocument{%
\patchBothAmsMathEnvironmentsForLineno{equation}%
\patchBothAmsMathEnvironmentsForLineno{align}%
\patchBothAmsMathEnvironmentsForLineno{flalign}%
\patchBothAmsMathEnvironmentsForLineno{alignat}%
\patchBothAmsMathEnvironmentsForLineno{gather}%
\patchBothAmsMathEnvironmentsForLineno{multline}%
}
\usepackage[mathlines]{lineno}

\makeatletter
\def\@seccntformat#1{%
  \protect\textup{\protect\@secnumfont
    \hspace*{5mm}\ifnum\pdfstrcmp{subsection}{#1}=0 \bfseries\fi
    \csname the#1\endcsname
    \protect\@secnumpunct
  }%
}
\makeatother

\makeatletter
\def\section{\@startsection{section}{1}%
\z@{.7\linespacing\@plus\linespacing}{.5\linespacing}%
{\bfseries\normalsize\scshape\centering}}
\makeatother

\makeatletter
\renewcommand{\@secnumfont}{\bfseries}
\makeatother

\begin{document}

\title{Adaptive IGAFEM with optimal \revision{convergence} rates: T-splines}

\author{Gregor Gantner}
\address{\new{Korteweg-de Vries Institute for Mathematics}\\
\new{University of Amsterdam, P.O. Box 94248, 1090 GE Amsterdam, The Netherlands}\\
\new{G.Gantner@uva.nl}}

\author{Dirk Praetorius}
\address{Institute for Analysis and Scientific Computing\\
TU Wien, Wiedner Hauptstra\ss{}e 8-10, A-1040 Wien, Austria\\
Dirk.Praetorius@asc.tuwien.ac.at}

\maketitle


\begin{abstract}
We consider an  adaptive algorithm for finite element methods for the isogeometric analysis (IGAFEM) of elliptic (possibly non-symmetric) second-order partial differential equations.
We employ analysis-suitable T-splines of arbitrary odd degree on T-meshes generated by the refinement strategy of [Morgenstern, Peterseim, Comput.\ Aided Geom.\ Design~34 (2015)] in 2D \new{and}  [Morgenstern, SIAM J.\ Numer.\ Anal.~54 (2016)] in 3D. 
Adaptivity is driven by some weighted residual {\sl a~posteriori} error estimator.
We prove linear convergence of the error estimator (\new{which is equivalent to} the sum of energy error plus data oscillations) with optimal algebraic rates \new{with respect to the number of elements of the underlying mesh}.
\end{abstract}




\section{Introduction}

\subsection{Adaptivity in isogeometric analysis}
The central idea of isogeometric analysis (IGA)~\cite{pioneer,bible,approximation} is to use the same ansatz functions for the discretization of the partial differential equation (PDE) as for the representation of the problem geometry in computer aided design (CAD).
While the CAD standard for spline representation in a multivariate setting relies on tensor-product B-splines, several extensions of the B-spline model have emerged  to allow for adaptive refinement, e.g.,  (analysis-suitable) T-splines~\cite{szbn03,djs10,scott,beirao}, hierarchical splines~\cite{juttler,juttler2,vanderzee},  or LR-splines~\cite{lr,dokken}; see also \cite{comparison1,comparison2} for a  comparison of these approaches.
All these concepts have been studied via numerical experiments.
However, to the best of our knowledge,  the thorough mathematical analysis of adaptive isogeometric finite element methods (IGAFEM) is so far  restricted to hierarchical splines \cite{bg,bgcopy,igafem,garau16,bbgv19}.
Recently, linear convergence at optimal algebraic rate \new{with respect to the number of mesh elements} has been proved in \cite{bgcopy} for the refinement strategy of \cite{bg} based on truncated hierarchical B-splines \cite{juttler2}, and in our own work \cite{igafem} for a newly proposed refinement strategy based on standard hierarchical B-splines. 
In the latter work, we identified certain abstract properties for the underlying meshes, the mesh-refinement, and the finite element spaces that \new{imply well-posedness, reliability, and efficiency of a residual {\sl a~posteriori} error estimator and} guarantee linear convergence at optimal rate \new{for a related adaptive mesh-refining algorithm}.
\new{Moreover, in \cite{igafem} we} verified these properties in the case of hierarchical  splines.
We stress that  adaptivity is  well understood for standard FEM with globally continuous piecewise polynomials; see, e.g., \cite{doerfler,mns00,bdd,stevenson,ckns,ffp14,axioms}  for milestones on convergence and optimal convergence rates.
In the frame of adaptive isogeometric boundary element methods (IGABEM), we also mention our recent works \cite{fgp,resigabem,resigaconv,diss,gps19}.

\subsection{Model problem}\label{sec:model}
On the bounded Lipschitz domain $\Omega\subset\R^d$, $d\in\{ 2,3\}$, with initial mesh $\TT_0$  and for given $f\in L^2(\Omega)$ as well as  $\ff\in L^2(\Omega)^d$ with $\ff|_T\in \boldsymbol{H}({\rm div},T)$ for all $T\in\TT_0$, we consider a general second-order linear elliptic PDE in divergence form with homogenous Dirichlet boundary conditions
\begin{align}\label{eq:problem}
\begin{split}
\mathcal{L}u:=-\div(\AA\nabla u)+\bb\cdot\nabla u+cu&=f+\div \ff\quad \text{in }\Omega,\\
u&=0\quad\text{on }\partial\Omega.
\end{split}
\end{align}
We pose the following regularity assumptions on the coefficients:
$\AA(x)\in\R^{d\times d}_{\mathrm{sym}}$ is a symmetric and uniformly positive definite matrix with  $\AA\in L^\infty(\Omega)^{d\times d}$ and $\AA|_T\in W^{1,\infty}(T)$ for all $T\in\TT_0$.
The vector $\bb(x)\in\R^d$   and the scalar $c(x)\in\R$ satisfy that  $\bb,c\in L^\infty(\Omega)$.
We interpret  $\mathcal{L}$ in its weak form and define the corresponding bilinear form 
\begin{align}
\edual{w}{v}:=\int_\Omega \AA(x)\nabla w(x)\cdot\nabla v(x)+\bb(x)\cdot\nabla w(x) v(x)+c(x)w(x)v(x)\,dx. 
\end{align}
The bilinear form is continuous, i.e., it holds that
\begin{align}
\edual{w}{v}\le \big(\norm{\AA}{L^\infty(\Omega)}+\norm{\bb}{L^\infty(\Omega)}+\norm{c}{L^\infty(\Omega)}\big)\norm{w}{H^1(\Omega)}\norm{v}{H^1(\Omega)}
\text{ for all $v,w\in H^1(\Omega)$.}
\end{align}
Additionally, we suppose ellipticity of $\edual{\cdot}{\cdot}$ on $H^1_0(\Omega)$, i.e.,
\begin{align}\label{eq:ellipticity}
\edual{v}{v}\ge \Cell\norm{v}{H^1(\Omega)}^2\quad\text{for all }v\in H_0^1(\Omega).
\end{align}
Note that~\eqref{eq:ellipticity} is for instance satisfied if $\AA(x)$ is uniformly positive definite and if $\bb\in \boldsymbol{H}({\rm div},\Omega)$ with $-\frac12\,{\rm div}\,\bb(x)+c(x)\ge0$ almost everywhere in $\Omega$.

Overall, the boundary value  problem~\eqref{eq:problem} fits into the setting of the Lax--Milgram theorem and therefore admits a unique solution $u\in H_0^1(\Omega)$ to the weak formulation
\begin{align}
 \edual{u}{v} = \int_\Omega fv-\ff\cdot\nabla v\,dx
 \quad\text{for all }v\in H^1_0(\Omega).
\end{align}
Finally, we note that the additional regularity $\ff|_T\in \boldsymbol{H}({\rm div},T)$ and $\AA|_T\in W^{1,\infty}(T)$ for all $T\in\TT_0$  is only required for the well-posedness of the residual {\sl a~posteriori} error estimator; see Section~\ref{subsec:estimator}.

\subsection{Outline \& Contributions}
The remainder of this work is  organized  as follows: 
\def\Ceff{C_{\rm eff}}%
\def\Crel{C_{\rm rel}}%
\def\osc{{\rm osc}}%
Section~\ref{sec:hierarchical} recalls the definition of T-meshes and T-splines of arbitrary odd degree in the parameter domain (Section~\ref{subsec:parameter hsplines}) from \cite{beirao} for $d=2$ \new{and from} \cite{morgensternT2}\footnote{To be precise, we define T-splines for $d=3$ slightly different than \cite{morgensternT2}; see Section~\ref{section:basis} for details.} for $d=3$.
Moreover, it recalls corresponding refinement strategies (Section~\ref{subsec:concrete refinement}) from \cite{morgensternT1,morgensternT2}, derives a canonical basis for the T-spline space with homogeneous boundary conditions (Section~\ref{section:basis}), and transfers all the definitions to the physical domain $\Omega$ via some parametrization $\gamma:[0,1]^d\to \overline\Omega$ (Section~\ref{subsec:physical hsplines}).
Subsequently, \revision{we formulate} a standard adaptive algorithm (Algorithm~\ref{the algorithm}) of the form 
\begin{align} \label{eq:box-algorithm}
 \boxed{\texttt{solve}}
 \longrightarrow
 \boxed{\texttt{estimate}}
 \longrightarrow
 \boxed{\texttt{mark}}
\longrightarrow
 \boxed{\texttt{refine}}
\end{align}
driven by some residual {\sl a~posteriori} error estimator~\eqref{eq:eta}. 
For T-splines in 2D, this algorithm has already been investigated numerically in \cite{comparison2}.
Finally, our main result Theorem~\ref{thm:main} is presented. 
First, it states that the error estimator $\eta_\revision{\ell}$ associated with the FEM solution $U_\revision{\ell}\in\XX_\revision{\ell}\subset H^1_0(\Omega)$ is efficient and
reliable, i.e., there exist 
$\Ceff$,
$\Crel>0$ such that 
\begin{align}\label{intro:releff}
 \Ceff^{-1}\,\eta_\revision{\ell}
 \le \inf_{V_\revision{\ell}\in\XX_\revision{\ell}}\big(\norm{u-V_\revision{\ell}}{H^1(\Omega)} + \osc_\revision{\ell}(V_\revision{\ell})\big)
 \le 
 \norm{u-U_\revision{\ell}}{H^1(\Omega)} 
 + \osc_\revision{\ell}(U_\revision{\ell})
 \le \Crel\,\eta_\revision{\ell},
\end{align} 
\new{where $\eta_\ell$ denotes the error estimator in the $\ell$-th step of the adaptive algorithm and} $\osc_\revision{\ell}(\cdot)$ denotes \new{the corresponding} data oscillation terms (see \eqref{eq:osc}).
 Second, it states that Algorithm~\ref{the algorithm} leads to linear convergence with optimal rates in the spirit of~\cite{stevenson,ckns,axioms}:  There exist $C>0$ and $0<q<1$ such that
\begin{align}\label{intro:linear}
 \eta_{\ell+j} \le C\,q^j\,\eta_\ell
 \quad\text{for all }\ell,j\in\N_0.
\end{align}
Moreover, for sufficiently small marking parameters in Algorithm~\ref{the algorithm}, the estimator 
(\new{and thus equivalently also} the so-called total error $\norm{u-U_\ell}{H^1(\Omega)} + \osc_\ell(U_\ell)$; see~\eqref{intro:releff}) 
decays even with the optimal algebraic convergence rate with respect to the number of mesh elements, i.e., 
\begin{align}
\eta_\ell=\OO\big((\#\TT_\ell)^{-s}\big)\quad\text{for all }\ell\in\N_0,
\end{align}
whenever the rate $s>0$ is possible for optimally chosen meshes.
The proof of Theorem~\ref{thm:main} is postponed to Section~\ref{sec:proof}
\new{and is based on} abstract properties of the underlying meshes, the mesh-refinement, the finite element spaces,  and the oscillations  which \new{have been identified in \cite{igafem} and} imply (an abstract version of) Theorem~\ref{thm:main}.
In Section~\ref{sec:proof}, we briefly recapitulate these properties and verify them for \new{the present} T-spline setting.
The final Section~\ref{sec:generalizations} comments on possible extensions of Theorem~\ref{thm:main}.

\subsection{General notation}
\label{sec:general notation}
Throughout, $|\cdot|$ denotes the absolute value of scalars, the Euclidean norm of vectors in $\R^d$, \new{and}  the $d$-dimensional measure of a set in $\R^d$. 
Moreover, $\#$ denotes the cardinality of a  set as well as the multiplicity of a knot within a given knot vector.
We write $A\lesssim B$ to abbreviate $A \le CB$ with some generic constant $C > 0$, which is clear from the context. Moreover, $A \simeq  B$ abbreviates $A\lesssim B \lesssim A$. Throughout, mesh-related quantities have the same index, e.g., $\XX_\bullet$ is the ansatz space corresponding to the  mesh  $\TT_\bullet$. 
The analogous notation is used for meshes $\TT_\circ$, $\TT_\star$, $\TT_{\ell}$ etc.
Moreover, we use $\widehat{\cdot}$ to transfer  quantities in the physical domain $\Omega$ to the parameter domain $\widehat \Omega$, e.g., we write $\widehat\T$ for the set of all admissible meshes in the parameter domain, while  $\T$  denotes the set of all admissible meshes in the physical domain.


\section{Adaptivity with T-splines}\label{sec:hierarchical}
In this section, we recall the formal definition of T-splines from \cite{beirao} for $d=2$ \new{and from}  \cite{morgensternT2}
for $d=3$ as well as  corresponding mesh-refinement strategies from \cite{morgensternT1,morgensternT2}. 
\new{While the mathematically sound definition is a bit tedious, the basic idea of T-splines is simple: 
Given a rectangular mesh (with hanging nodes) as in Figure~\ref{fig:Tmesh}, one associates to all nodes a local knot vector in each direction as the intersections (projected into the white area $\overline{\widehat\Omega}$) of the line in that direction through the node (indicated in red) with the mesh skeleton.
The resulting local knot vectors then induce a standard tensor-product B-spline, and the set of all such B-splines spans the corresponding T-spline space.}
\new{Moreover,} we formulate an adaptive algorithm (Algorithm~\ref{the algorithm}) for conforming FEM discretizations  of our model problem~\eqref{eq:problem}, where adaptivity is driven by the residual {\sl a~posteriori} error estimator (see \eqref{eq:eta} below).
Our main result of the present work Theorem~\ref{thm:main}  states reliability and efficiency of the estimator as well as linear convergence at optimal algebraic rate \new{with respect to the number of mesh elements.}

\subsection{T-meshes and T-splines in the parameter domain $\bold{\widehat{\Omega}}$}\label{subsec:parameter hsplines}
Meshes $\TT_\coarse$ and corresponding spaces $\XX_\coarse$ are defined through their counterparts on the \textit{parameter domain}
\begin{align}
\widehat\Omega:=(0,N_1)\times\dots\times(0,N_d),
\end{align} where $N_i\in\N$ are fixed integers for $i\in\{1,\dots,d\}$.
\revision{Recall that the symbol $\bullet$ is used for a generic mesh index to relate corresponding quantities; see the general notation of Section~\ref{sec:general notation}.}
Let $p_1,\dots,p_d\ge 3$ be  fixed odd polynomial degrees.
Let $\widehat\TT_0$ be an initial uniform tensor-mesh of the form
\begin{align}\label{eq:TT0}
\widehat\TT_0=\Big\{\prod_{i=1}^d[a_i,a_i+1]:a_i\in\{0,\dots,N_i-1\}\text{ for }i\in\{1,\dots,d\}\Big\}.
\end{align}
For an arbitrary hyperrectangle $\widehat T=[a_1,b_1]\times\dots[a_d,b_d]$, we define its bisection in direction $i\in\{1,\dots,d\}$ as the set 
\begin{align}
\begin{split}
\bisect_i(\widehat T):=
\Big\{&\prod_{j=1}^{i-1}[a_j,b_j]\times\Big[a_i,\frac{a_i+b_i}{2}\Big]\times\prod_{j=i+1}^d[a_j,b_j],\\
&\prod_{j=1}^{i-1}[a_j,b_j]\times\Big[\frac{a_i+b_i}{2},b_i\Big]\times\prod_{j=i+1}^d[a_j,b_j]\Big\}.
\end{split}
\end{align}
For $k\in\N_0$, let 
\begin{align}\label{eq:underline}
\k:=\lfloor k/d\rfloor d
\end{align}
 and define the $k$-th uniform refinement of $\widehat\TT_0$ inductively by
\begin{align}
\widehat\TT_{\uni{0}}:=\widehat\TT_0\quad\text{and}\quad
\widehat\TT_{\uni{k}}:=\bigcup\set{\bisect_{k+1-\k}(\widehat T)}{\widehat T\in\widehat\TT_{\uni{k-1}}};
\end{align}
see Figure~\ref{fig:bisect2d} and \ref{fig:bisect3d}.
\new{Note that the direction of bisection changes periodically.}

\begin{figure}
\begin{tikzpicture}[scale=2]
    \coordinate (O) at (0,0);
    \draw[line width=0.5mm, color=black] (O) -- (1,0);
    \draw[line width=0.5mm, color=black] (1,0) -- (1,1);
    \draw[line width=0.5mm, color=black] (1,1) -- (0,1);
    \draw[line width=0.5mm, color=black] (0,1) -- (O);        
    \node[below] at (current bounding box.south) {$\widehat\TT_{\uni{0}}$};
\end{tikzpicture}
\qquad
\begin{tikzpicture}[scale=2]
    \coordinate (O) at (0,0);
    \draw[line width=0.5mm, color=black] (O) -- (1,0);
    \draw[line width=0.5mm, color=black] (1,0) -- (1,1);
    \draw[line width=0.5mm, color=black] (1,1) -- (0,1);
    \draw[line width=0.5mm, color=black] (0,1) -- (O);        
    \draw[line width=0.5mm, color=black] (0.5,0) -- (0.5,1);        
    \node[below] at (current bounding box.south) {$\widehat\TT_{\uni{1}}$};
\end{tikzpicture}
\qquad
\begin{tikzpicture}[scale=2]
    \coordinate (O) at (0,0);
    \draw[line width=0.5mm, color=black] (O) -- (1,0);
    \draw[line width=0.5mm, color=black] (1,0) -- (1,1);
    \draw[line width=0.5mm, color=black] (1,1) -- (0,1);
    \draw[line width=0.5mm, color=black] (0,1) -- (O);        
    \draw[line width=0.5mm, color=black] (0.5,0) -- (0.5,1);        
    \draw[line width=0.5mm, color=black] (0,0.5) -- (1,0.5);  
    \node[below] at (current bounding box.south) {$\widehat\TT_{\uni{2}}$};
\end{tikzpicture}
\qquad
\begin{tikzpicture}[scale=2]
    \coordinate (O) at (0,0);
    \draw[line width=0.5mm, color=black] (O) -- (1,0);
    \draw[line width=0.5mm, color=black] (1,0) -- (1,1);
    \draw[line width=0.5mm, color=black] (1,1) -- (0,1);
    \draw[line width=0.5mm, color=black] (0,1) -- (O);        
    \draw[line width=0.5mm, color=black] (0.5,0) -- (0.5,1);        
    \draw[line width=0.5mm, color=black] (0,0.5) -- (1,0.5);        
    \draw[line width=0.5mm, color=black] (0.25,0) -- (0.25,1);        
    \draw[line width=0.5mm, color=black] (0.75,0) -- (0.75,1);  
    \node[below] at (current bounding box.south) {$\widehat\TT_{\uni{3}}$};      
\end{tikzpicture}
\qquad
\begin{tikzpicture}[scale=2]
    \coordinate (O) at (0,0);
    \draw[line width=0.5mm, color=black] (O) -- (1,0);
    \draw[line width=0.5mm, color=black] (1,0) -- (1,1);
    \draw[line width=0.5mm, color=black] (1,1) -- (0,1);
    \draw[line width=0.5mm, color=black] (0,1) -- (O);        
    \draw[line width=0.5mm, color=black] (0.5,0) -- (0.5,1);        
    \draw[line width=0.5mm, color=black] (0,0.5) -- (1,0.5);        
    \draw[line width=0.5mm, color=black] (0.25,0) -- (0.25,1);        
    \draw[line width=0.5mm, color=black] (0.75,0) -- (0.75,1);      
    \draw[line width=0.5mm, color=black] (0,0.25) -- (1,0.25);      
    \draw[line width=0.5mm, color=black] (0,0.75) -- (1,0.75);    
    \node[below] at (current bounding box.south) {$\widehat\TT_{\uni{4}}$};          
\end{tikzpicture}

\caption{
Uniform refinements of 2D initial partition $\widehat\TT_0:=\{[0,1]^2\}$.}
\label{fig:bisect2d}
\end{figure}
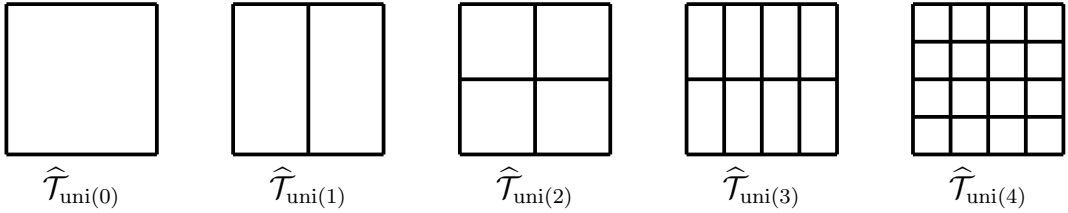

\begin{figure}
\tdplotsetmaincoords{70}{20}
\begin{tikzpicture}[scale=2,tdplot_main_coords]
    \coordinate (O) at (0,0,0);
    \draw[line width=0.5mm, color=black] (O) -- (1,0,0);
    \draw[line width=0.5mm, color=black] (O) -- (0,0,1);
    \draw[line width=0.5mm, color=black] (1,0,0) -- (1,0,1);
    \draw[line width=0.5mm, color=black] (1,0,1) -- (1,1,1);
    \draw[line width=0.5mm, color=black] (0,0,1) -- (1,0,1);
    \draw[line width=0.5mm, color=black] (1,0,0) -- (1,1,0);
    \draw[line width=0.5mm, color=black] (1,1,0) -- (1,1,1);
    \draw[line width=0.5mm, color=black] (0,1,1) -- (1,1,1);
    \draw[line width=0.5mm, color=black] (0,0,1) -- (0,1,1);      
    \node[below] at (current bounding box.south) {$\widehat\TT_{\uni{0}}$};  
\end{tikzpicture}
\quad
\begin{tikzpicture}[scale=2,tdplot_main_coords]
    \coordinate (O) at (0,0,0);
    \draw[line width=0.5mm, color=black] (O) -- (1,0,0);
    \draw[line width=0.5mm, color=black] (O) -- (0,0,1);
    \draw[line width=0.5mm, color=black] (1,0,0) -- (1,0,1);
    \draw[line width=0.5mm, color=black] (1,0,1) -- (1,1,1);
    \draw[line width=0.5mm, color=black] (0,0,1) -- (1,0,1);
    \draw[line width=0.5mm, color=black] (1,0,0) -- (1,1,0);
    \draw[line width=0.5mm, color=black] (1,1,0) -- (1,1,1);
    \draw[line width=0.5mm, color=black] (0,1,1) -- (1,1,1);
    \draw[line width=0.5mm, color=black] (0,0,1) -- (0,1,1);
    \draw[line width=0.5mm, color=black] (0.5,0,0) -- (0.5,0,1);
    \draw[line width=0.5mm, color=black] (0.5,0,1) -- (0.5,1,1);    
    \node[below] at (current bounding box.south) {$\widehat\TT_{\uni{1}}$};
\end{tikzpicture}
\quad
\begin{tikzpicture}[scale=2,tdplot_main_coords]
    \coordinate (O) at (0,0,0);
    \draw[line width=0.5mm, color=black] (O) -- (1,0,0);
    \draw[line width=0.5mm, color=black] (O) -- (0,0,1);
    \draw[line width=0.5mm, color=black] (1,0,0) -- (1,0,1);
    \draw[line width=0.5mm, color=black] (1,0,1) -- (1,1,1);
    \draw[line width=0.5mm, color=black] (0,0,1) -- (1,0,1);
    \draw[line width=0.5mm, color=black] (1,0,0) -- (1,1,0);
    \draw[line width=0.5mm, color=black] (1,1,0) -- (1,1,1);
    \draw[line width=0.5mm, color=black] (0,1,1) -- (1,1,1);
    \draw[line width=0.5mm, color=black] (0,0,1) -- (0,1,1);
    \draw[line width=0.5mm, color=black] (0.5,0,0) -- (0.5,0,1);
    \draw[line width=0.5mm, color=black] (0.5,0,1) -- (0.5,1,1);    
    \draw[line width=0.5mm, color=black] (1,0.5,0) -- (1,0.5,1);    
    \draw[line width=0.5mm, color=black] (1,0.5,1) -- (0,0.5,1);    
    \node[below] at (current bounding box.south) {$\widehat\TT_{\uni{2}}$};
\end{tikzpicture}
\quad
\begin{tikzpicture}[scale=2,tdplot_main_coords]
    \coordinate (O) at (0,0,0);
    \draw[line width=0.5mm, color=black] (O) -- (1,0,0);
    \draw[line width=0.5mm, color=black] (O) -- (0,0,1);
    \draw[line width=0.5mm, color=black] (1,0,0) -- (1,0,1);
    \draw[line width=0.5mm, color=black] (1,0,1) -- (1,1,1);
    \draw[line width=0.5mm, color=black] (0,0,1) -- (1,0,1);
    \draw[line width=0.5mm, color=black] (1,0,0) -- (1,1,0);
    \draw[line width=0.5mm, color=black] (1,1,0) -- (1,1,1);
    \draw[line width=0.5mm, color=black] (0,1,1) -- (1,1,1);
    \draw[line width=0.5mm, color=black] (0,0,1) -- (0,1,1);     
    \draw[line width=0.5mm, color=black] (0.5,0,0) -- (0.5,0,1);
    \draw[line width=0.5mm, color=black] (0.5,0,1) -- (0.5,1,1);    
    \draw[line width=0.5mm, color=black] (1,0.5,0) -- (1,0.5,1);    
    \draw[line width=0.5mm, color=black] (1,0.5,1) -- (0,0.5,1);    
    \draw[line width=0.5mm, color=black] (0,0,0.5) -- (1,0,0.5);    
    \draw[line width=0.5mm, color=black] (1,0,0.5) -- (1,1,0.5);    
    \node[below] at (current bounding box.south) {$\widehat\TT_{\uni{3}}$};
\end{tikzpicture}
\quad
\begin{tikzpicture}[scale=2,tdplot_main_coords]
    \coordinate (O) at (0,0,0);
    \draw[line width=0.5mm, color=black] (O) -- (1,0,0);
    \draw[line width=0.5mm, color=black] (O) -- (0,0,1);
    \draw[line width=0.5mm, color=black] (1,0,0) -- (1,0,1);
    \draw[line width=0.5mm, color=black] (1,0,1) -- (1,1,1);
    \draw[line width=0.5mm, color=black] (0,0,1) -- (1,0,1);
    \draw[line width=0.5mm, color=black] (1,0,0) -- (1,1,0);
    \draw[line width=0.5mm, color=black] (1,1,0) -- (1,1,1);
    \draw[line width=0.5mm, color=black] (0,1,1) -- (1,1,1);
    \draw[line width=0.5mm, color=black] (0,0,1) -- (0,1,1);     
    \draw[line width=0.5mm, color=black] (0.5,0,0) -- (0.5,0,1);
    \draw[line width=0.5mm, color=black] (0.5,0,1) -- (0.5,1,1);    
    \draw[line width=0.5mm, color=black] (1,0.5,0) -- (1,0.5,1);    
    \draw[line width=0.5mm, color=black] (1,0.5,1) -- (0,0.5,1);    
    \draw[line width=0.5mm, color=black] (0,0,0.5) -- (1,0,0.5);    
    \draw[line width=0.5mm, color=black] (1,0,0.5) -- (1,1,0.5);    
    \draw[line width=0.5mm, color=black] (0.25,0,0) -- (0.25,0,1);    
    \draw[line width=0.5mm, color=black] (0.75,0,0) -- (0.75,0,1);    
    \draw[line width=0.5mm, color=black] (0.25,0,1) -- (0.25,1,1);    
    \draw[line width=0.5mm, color=black] (0.75,0,1) -- (0.75,1,1);   
    \node[below] at (current bounding box.south) {$\widehat\TT_{\uni{4}}$}; 
\end{tikzpicture}
\caption{Uniform refinements of 3D initial partition $\widehat\TT_0:=\{[0,1]^3\}$.}
\label{fig:bisect3d}
\end{figure}

A finite set $\widehat\TT_\coarse$ is a \emph{T-mesh (in the parameter domain)}, if $\widehat\TT_\coarse \subseteq\bigcup_{k\in\N_0} \widehat\TT_{\uni{k}}$, $\bigcup_{\widehat T\in\widehat\TT_\coarse} \widehat T=\overline{\widehat\Omega}$, and $|\widehat T\cap\widehat T'|=0$ for all $\widehat T,\widehat T'\in\widehat\TT_\coarse$ with $\widehat T\neq\widehat T'$.
 For an illustrative example of a \new{general} T-mesh, see Figure~\ref{fig:Tmesh}.
Since $\widehat\TT_{\uni{k}}\cap\widehat\TT_{\uni{k'}}=\emptyset$ for $k,k'\in\N_0$ with $k\neq k'$, each element $\widehat T\in\widehat\TT_\coarse$ has a natural level
\begin{align}
\level(\widehat T):=k\in\N_0 \quad\text{with }\widehat T\in\widehat\TT_{\uni{k}}.
\end{align}
 
 \begin{figure}[t] 
\begin{center}
\includegraphics[width=1\textwidth]{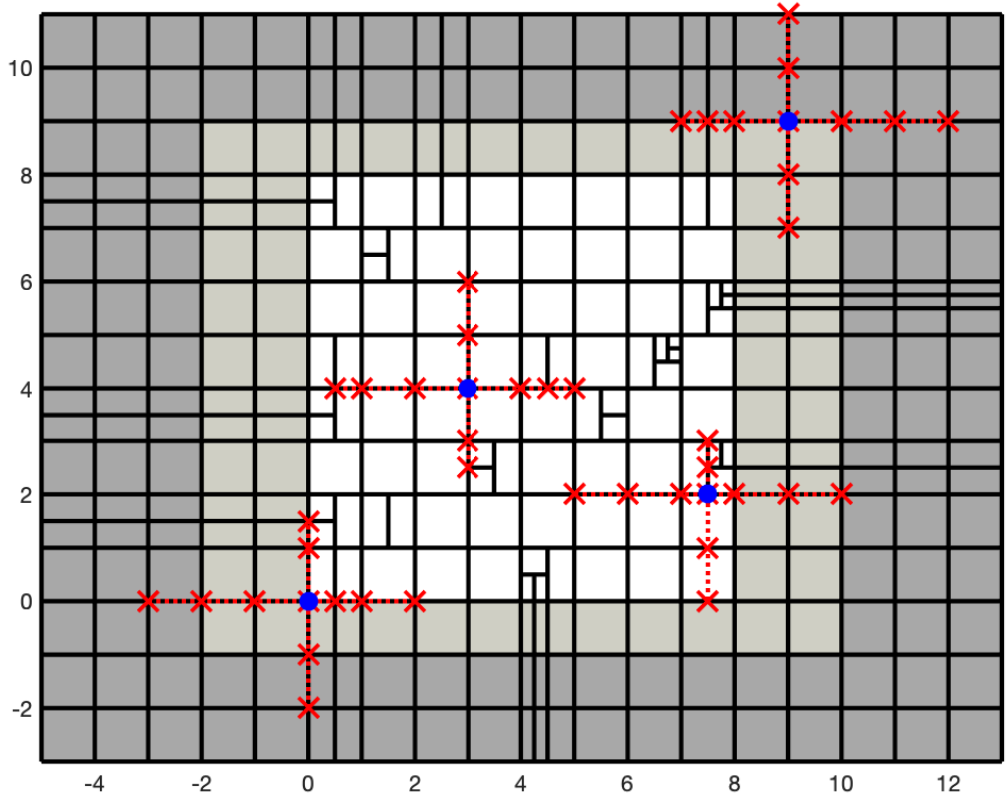}
\end{center}
\caption{
A \new{general (non-admissible)} T-mesh $\widehat\TT_\bullet^\ex$ in 2D with $(p_1,p_2)=(5,3)$ is depicted.
The sets $\widehat\Omega$, $\widehat\Omega^\act$, and $\widehat\Omega^\ex$ are highlighted in white, light gray, and dark gray, respectively.
For the three (blue) nodes $z\in\{(0,0),(3,4),\new{(7.5,2)},(9,9)\}$, their corresponding local index vectors $\widehat\II_{\coarse,i}^{\rm loc}(z)$ with $i\in\{1,2\}$  are indicated by (red) crosses.
\new{We also indicate in red the lines through the nodes mentioned at the beginning of  Section~\ref{sec:hierarchical}. The local knot vectors are obtained by setting all negative values to $0$ and all values larger than $8$ to $8$, i.e., 
$\widehat\KK_{\coarse,1}^{\rm loc}\big((0,0)\big)=(0,0,0,0,0.5,1,2)$, 
$\widehat\KK_{\coarse,2}^{\rm loc}\big((0,0)\big)=(0,0,0,1,1.5)$,
$\widehat\KK_{\coarse,1}^{\rm loc}\big((3,4)\big)=(0.5,1,2,3,4,4.5,5)$, 
$\widehat\KK_{\coarse,2}^{\rm loc}\big((3,4)\big)=(2.5,3,4,5,6)$, 
$\widehat\KK_{\coarse,1}^{\rm loc}\big((7.5,2)\big)=(5,6,7,7.5,8,8,8)$, 
$\widehat\KK_{\coarse,2}^{\rm loc}\big((7.5,2)\big)=(0,1,2,2.5,3)$,
$\widehat\KK_{\coarse,1}^{\rm loc}\big((9,9)\big)=(7,7.5,8,8,8,8,8)$, 
$\widehat\KK_{\coarse,2}^{\rm loc}\big((9,9)\big)=(7,8,8,8,8)$.
This guarantees that all associated B-splines have support within the closure of~$\widehat\Omega$.}
}
\label{fig:Tmesh}
\end{figure}

 In order to define T-splines, \new{in particular to allow for knot multiplicities larger than one at the boundary,}  we have to extend the mesh $\widehat\TT_\coarse$ on $\overline{\widehat\Omega}$ to a mesh on $\overline{\widehat\Omega^\ex}$, where
 \begin{align}
 \widehat\Omega^\ex:=\prod_{i=1}^d(-p_i,N_i+p_i).
  \end{align}
We define $\widehat\TT_0^\ex$ analogously to \eqref{eq:TT0} and $\widehat\TT_\coarse^\ex$ as the mesh on $\overline{\widehat\Omega^\ex}$ that is obtained by extending any bisection, \new{which} takes place on the boundary $\partial\widehat\Omega$ during the refinement from $\widehat\TT_0$ to $\widehat\TT_\coarse$, to the set $\overline{\widehat\Omega^\ex}\setminus\overline{\widehat\Omega}$; see Figure~\ref{fig:Tmesh}.
For $d=2$, this \new{formally} reads
\begin{align}
\widehat\TT_\coarse^\ex:=\widehat\TT_\coarse\cup\set{{\rm ext}_\coarse(\widehat E_1\times \widehat E_2)}{{\rm dim}(\widehat E_1\times \widehat E_2)<2\wedge \widehat E_i\in\{\{0\},\{N_i\},[0,N_i]\}\text{ for }i\in\{1,2\}}, 
 \end{align}
 where 
\begin{align*}
 {\rm ext}_\coarse\big(\{(0,0)\}\big)&:=\set{[a_1,a_1+1]\times[a_2,a_2+1]}{a_i\in\{-p_i,\dots,-1\}\text{ for }i\in\{1,2\}},\\
{\rm ext}_\coarse\big([0,N_1]\times\{0\})\big)&:= \set{[a_1,b_1]\times[a_2,a_2+1]}{a_2\in\{-p_2,\dots,-1\}\wedge\exists b_2':[a_1,b_1]\times[0,b_2']\in\widehat\TT_\coarse},
\end{align*} 
and the remaining ${\rm ext}_\coarse(\cdot)$ terms are defined analogously.
Note that the logical expression 
\begin{align*}
\exists b_2':[a_1,b_1]\times[0,b_2']\in\widehat\TT_\coarse
\end{align*}
 means that there exists an element at the (lower part of the) boundary $\partial\widehat\Omega$ with side $[a_1,b_1]$.
 For $d=3$, this reads 
 \begin{align*}
\widehat\TT_\coarse^\ex:=\widehat\TT_\coarse \cup\Big\{{\rm ext}_\coarse\Big({\prod_{i=1}^3}\widehat E_i\Big):{{\rm dim}\Big(\prod_{i=1}^3}\widehat E_i\Big)<3\wedge \widehat E_i\in\{\{0\},\{N_i\},[0,N_i]\}\text{ for }i\in\{1,2,3\}\Big\},
 \end{align*}
 where 
 \begin{align*}
& {\rm ext}_\coarse\big(\{(0,0,0)\}\big):=\Big\{\prod_{i=1}^3[a_i,a_i+1]: a_i\in\{-p_i,\dots,-1\}\text{ for }i\in\{1,2,3\}\Big\},\\
&{\rm ext}_\coarse\big([0,N_1]\times\{0\}\times\{0\}\big):= \big\{[a_1,b_1]\times[a_2,a_2+1]\times[a_3,a_3+1]:\\
&\qquad a_i\in\{-p_i,\dots,-1\}\text{ for }i\in\{2,3\}\wedge\exists b_2',b_3':[a_1,b_1]\times[0,b_2']\times[0,b_3']\in\widehat\TT_\coarse\big\},\\
&{\rm ext}_\coarse\big([0,N_1]\times[0,N_2]\times\{0\}\big):= \big\{[a_1,b_1]\times[a_2,b_2]\times[a_3,a_3+1]:\\
&\qquad a_3\in\{-p_3,\dots,-1\}\wedge\exists b_3':[a_1,b_1]\times[a_2,b_2]\times[0,b_3']\in\widehat\TT_\coarse\big\},
\end{align*} 
and the remaining ${\rm ext}_\coarse(\cdot)$ terms are defined analogously.
Note that the logical expressions
\begin{align*}
\exists b_2',b_3':[a_1,b_1]\times[0,b_2']\times[0,b_3']\in\widehat\TT_\coarse
\quad\text{\new{and}}\quad \exists b_3':[a_1,b_1]\times[a_2,b_2]\times[0,b_3']\in\widehat\TT_\coarse
\end{align*}
 mean that there exists an element at the (lower part of the) boundary $\partial\widehat\Omega$ with side $[a_1,b_1]$, \new{and} with sides $[a_1,b_1]$ \new{as well as} $[a_2,b_2]$, \new{respectively}.
The corresponding \textit{skeleton} in any direction $i\in\{1,\dots,d\}$ reads
 \begin{align}
 \partial_i \widehat\TT^\ex_\coarse:=\bigcup\big\{\prod_{j=1}^{i-1}[a_j,b_j]\times\{a_i,b_i\}\times\prod_{j=i+1}^d[a_j,b_j]: \prod_{j=1}^d[a_j,b_j]\in\widehat\TT^\ex_\coarse\big\}.
 \end{align}
Recall that $p_i\ge 3$ are odd.
We abbreviate 
\begin{align}\label{eq:omega1}
\widehat\Omega^\act:=\prod_{i=1}^d\big(-(p_i-1)/2,N_i+(p_i-1)/2\big).
\end{align}
As in the literature, its closure $\overline{\widehat \Omega^\act}$ is called \textit{active region}, whereas   $\overline{\widehat\Omega^\ex}\setminus\widehat\Omega^\act$ is called \textit{frame region}. 
The set of \textit{nodes} $\widehat\NN^\act_\coarse$ in the active region reads 
\begin{align}\label{eq:nodes1}
\widehat\NN_\coarse^\act:=\set{z\in\overline{\widehat\Omega^\act}}{z\text{ is vertex of some }\widehat T\in\widehat\TT_\coarse^\act}.
\end{align}
To each node $z=(z_1,\dots,z_d)\in\widehat\NN^\act_\coarse$ and each direction $i\in\{1,\dots,d\}$, we associate the corresponding \textit{global index vector}
\new{which is obtained by drawing a line in the $i$-th direction through the node $z$ and collecting the $i$-th coordinates of the intersections with the skeleton.
Formally, this reads}
\begin{align*}
\widehat\II_{\coarse,i}^{\rm gl}(z):={\rm sort}\big(\set{t\in[-p_i,N_i+p_i]}{(z_1,\dots,z_{i-1},t,z_{i+1},\dots,z_d)\in\partial_i\widehat\TT^\ex_\coarse}\big),
\end{align*}
where ${\rm sort}(\cdot)$ returns (in ascending order) the sorted vector corresponding to a set of numbers.
The corresponding \textit{local index vector} 
\begin{align}
\widehat\II^{\rm loc}_{\coarse,i}(z)\in\R^{p_i+2}
\end{align}
is the  vector of all $p_i+2$ consecutive elements in $\widehat\II_{\coarse,i}^{\rm gl}(z)$ having $z_i$ as their $((p_i+3)/2)$-th (i.e., their middle) entry; see Figure~\ref{fig:Tmesh}.
Note that such elements always exist due to the definition of $\widehat\II_{\coarse,i}^{\rm gl}(z)$ and the fact that $p_i$ is odd.
This induces the \textit{global knot vector}
\begin{align}
\widehat\KK_{\coarse,i}^{\rm gl}(z):=\max\Big(\min\big(\widehat\II_{\coarse,i}^{\rm gl}(z),N_i\big),0\Big),
\end{align}
and the \textit{local knot vector}
\begin{align}\label{eq:local knot vector}
\widehat\KK_{\coarse,i}^{\rm loc}(z):=\max\Big(\min\big(\widehat\II_{\coarse,i}^{\rm loc}(z),N_i\big),0\Big), 
\end{align}
where $\max(\cdot,0)$ and $\min(\cdot,N_i)$ are understood element-wise (i.e., for each element in $\widehat \II_{\coarse,i}^{\rm gl}(z)$ \new{and} $\widehat \II_{\coarse,i}^{\rm \new{loc}}(z)$, \new{respectively}).
We stress that the resulting global knot vectors in each direction are so-called \emph{open knot vectors}, i.e., the multiplicity of the first knot $0$ and the last knot $N_i$ is $p_i+1$.
Moreover, the interior knots coincide with the indices in $\widehat\Omega$ and all have multiplicity one.
For more general index to parameter mappings, we refer to Section~\ref{subsec:non-uniform}.
We define the corresponding tensor-product B-spline $\widehat B_{\coarse,z}:\overline{\widehat \Omega}\to \R$ as 
\begin{align}\label{eq:Bz}
\widehat B_{\coarse,z}(t_1,\dots,t_d):=\prod_{i=1}^d \widehat B\big(t_i\,|\,\widehat\KK^{\rm loc}_{\coarse,i}(z)\big)\quad\text{for \new{all} }(t_1,\dots,t_d)\in\overline{\widehat\Omega}, 
\end{align}
where $\widehat B\big(t_i\,|\,\widehat\KK^{\rm loc}_{\coarse,i}(z)\big)$ denotes the unique one-dimensional B-spline induced by $\widehat\KK^{\rm loc}_{\coarse,i}(z)$.
\new{For convenience of the reader, we recall the following definition for arbitrary $p\in\N_0$ via divided differences\footnote{\new{For any function $F:\R\to\R$, divided differences are recursively defined via $[x_0]F:=F(x_0)$ and $[x_0,\dots,x_{j+1}]F:=\big([x_1,\dots,x_{j+1}]F - [x_0,\dots,x_j]F\big) / (x_{j+1}-x_0)$ for $j=0,\dots,p$.}},
\begin{align*}
\widehat B\big(t|(x_0,\dots,x_{p+1})\big) := (x_{p+1}-x_0) \cdot [x_0,\dots,x_{p+1}] \big(\max\{(\cdot)-t,0\}^p\big)
\\ \text{for \new{all} } t\in\R\text{ and } x_0\le\dots\le x_{p+1};
\end{align*}
see also \cite{boor} for equivalent definitions and elementary properties.
We only mention that $\widehat B\big(\cdot|(x_0,\dots,x_{p+1})\big)$ is positive on the open interval $(x_0,x_{p+1})$, it does not vanish at $x_0$  if and only if $x_0=\dots=x_{p}$, and it does not vanish at $x_{p+1}$ if and only if $x_1=\dots=x_{p+1}$. 
Due to definition~\eqref{eq:local knot vector}, each $\widehat B_{\coarse,z}$ has thus indeed only support within the closure of the parameter domain $\widehat\Omega$, and multiple knots may only occur at the boundary $\partial\widehat\Omega$.}
According to, e.g., \cite[Section~6]{boor},   each tensor-product B-spline satisfies that $\widehat B_{\coarse,z}\in C^2\big(\overline{\widehat\Omega}\big)$.
With this, we see for the space of \textit{T-splines in the parameter domain} that 
\begin{align}\label{eq:Tsplines C2}
\widehat\YY_\bullet:={\rm span}\set{\widehat B_{\coarse,z}}{z\in\widehat\NN^\act_\coarse}\subset C^2\big(\overline{\widehat\Omega}\big).
\end{align}
Finally, we define our ansatz space in the parameter domain as 
\begin{align}
\widehat\XX_\bullet:=\set{\widehat V_\bullet\in\widehat\YY_\bullet}{\widehat V_\bullet|_{\partial\widehat\Omega}=0}. 
\end{align}
Note that this specifies the abstract setting of Section~\ref{subsec:ansatz}.
 For a more detailed introduction to T-meshes and splines, we refer to, e.g., \cite[Section~7]{variational}.

\subsection{Refinement in the parameter domain $\bold{\widehat\Omega}$}\label{subsec:concrete refinement}
In this section, we recall the refinement algorithm from \cite[Algorithm~2.9 and Corollary~2.15]{morgensternT1} for $d=2$ and \cite[Algorithm~2.9]{morgensternT2} for $d=3$; see also \cite[Chapter~5]{diss_morgenstern}.
To this end, we first define for a T-mesh $\widehat\TT_\coarse$ and $\widehat T\in\widehat\TT_\bullet$ with $k:=\level(\widehat T)$ the set of 
its \textit{neighbors}
\begin{align}\label{eq:neighbors}
\boldsymbol{\rm N}_\bullet(\widehat T):=\set{\widehat T'\in\widehat\TT_\coarse}{\exists t\in\widehat T'\text{ with }|{\rm mid}_i(\widehat T)-t_i|<D_i(k)\text{ for all }i\in\{1,\dots,d\}},
\end{align}
where ${\rm mid}(\widehat T)=({\rm mid}_1(\widehat T),\dots,{\rm mid}_d(\widehat T))$ denotes the midpoint of $\widehat T$ and $D(k)=(D_1(k),\dots,D_d(k))$ is defined as 
\begin{align}\label{eq:D defined}
D(k):=\begin{cases}
2^{-k/2}\big(p_1/2,\,p_2/2+1\big)\,&\text{if }d=2\text{ and }k=0\,\,{\rm mod}\,\,2,\\
2^{-(k\new{-}1)/2}\big(p_1/\new{4}+1\new{/2},\,p_2\new{/2}\big)\,&\text{if }d=2\text{ and }k=1\,\,{\rm mod}\,\,2,\\
2^{-k/3}\big(p_1+3/2,\,p_2+3/2,\,p_3+3/2\big)\,&\text{if }d=3\text{ and }k=0\,\,{\rm mod}\,\,3,\\
2^{-(k-1)/3}\big(p_1/2+3/4,\,p_2+3/2,\,p_3+3/2\big)\,&\text{if }d=3\text{ and }k=1\,\,{\rm mod}\,\,3,\\
2^{-(k-2)/3}\big(p_1/2+3/4,\,p_2/2+3/4,\,p_3+3/2\big)\,&\text{if }d=3\text{ and }k=2\,\,{\rm mod}\,\,3;
\end{cases}
\end{align}
\new{see Figure~\ref{fig:Tmesh_uni} and \ref{fig:Tmeshes}
 for some examples.
For $d=2$, \cite[Corollary~2.15]{morgensternT1} also provides the following identity
\begin{align}
\boldsymbol{\rm N}_\bullet(\widehat T)=\set{\widehat T'\in\widehat\TT_\coarse}{|{\rm mid}_i(\widehat T)- {\rm mid}_i(\widehat T')|\le D_i(k)\text{ for all }i\in\{1,2\}}.
\end{align}
}
We define the set of \textit{bad neighbors}
\begin{align}\label{eq:bad neighbors}
\boldsymbol{\rm N}^{\rm bad}_\bullet(\widehat T)&:=\set{\widehat T'\in\boldsymbol{\rm N}_\bullet(\widehat T)}{\level(\widehat T')<\level(\widehat T)}.
\end{align}

 \begin{figure}[t] 
\begin{center}
\includegraphics[width=0.49\textwidth,clip=true]{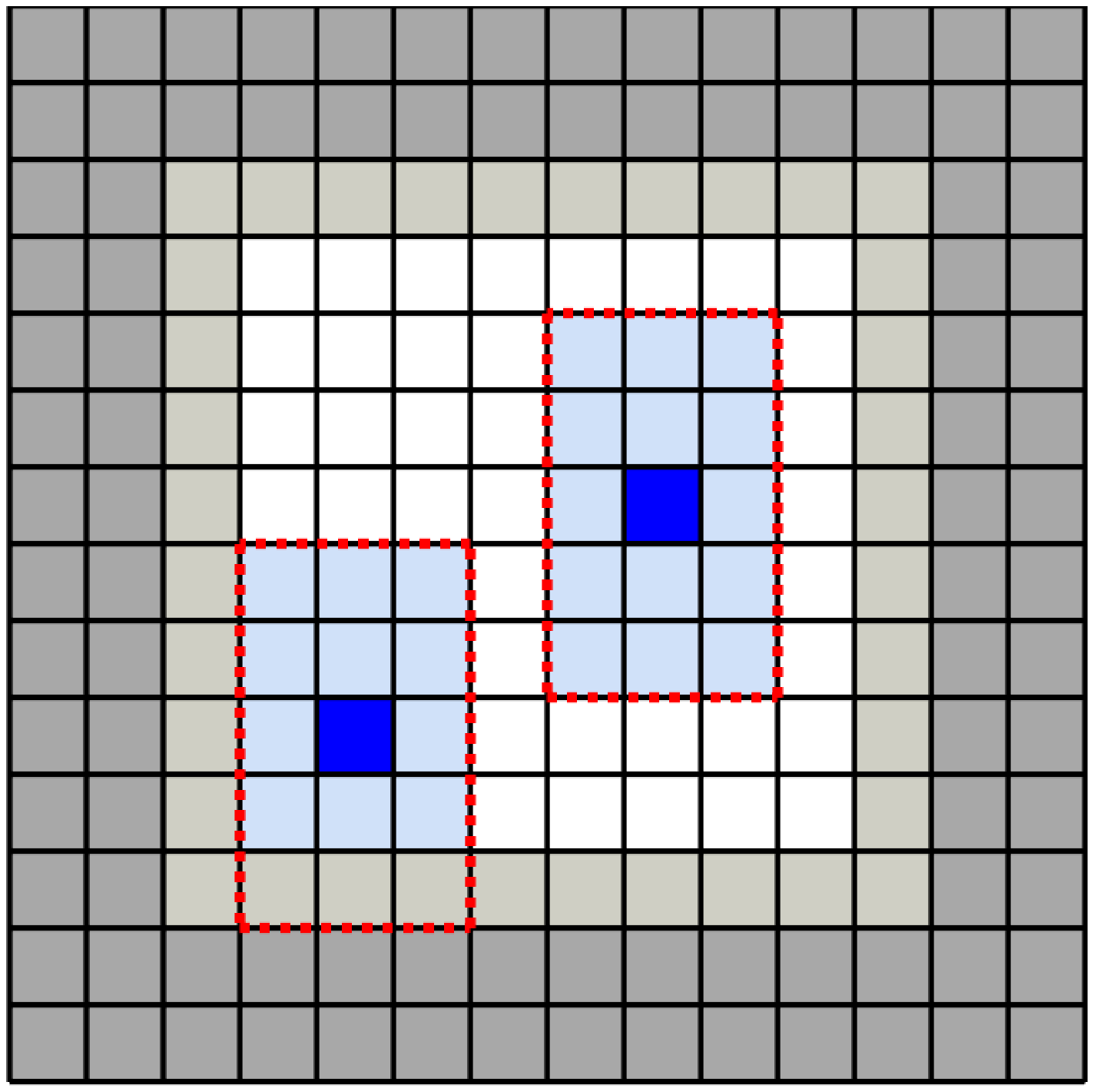}
\includegraphics[width=0.49\textwidth,clip=true]{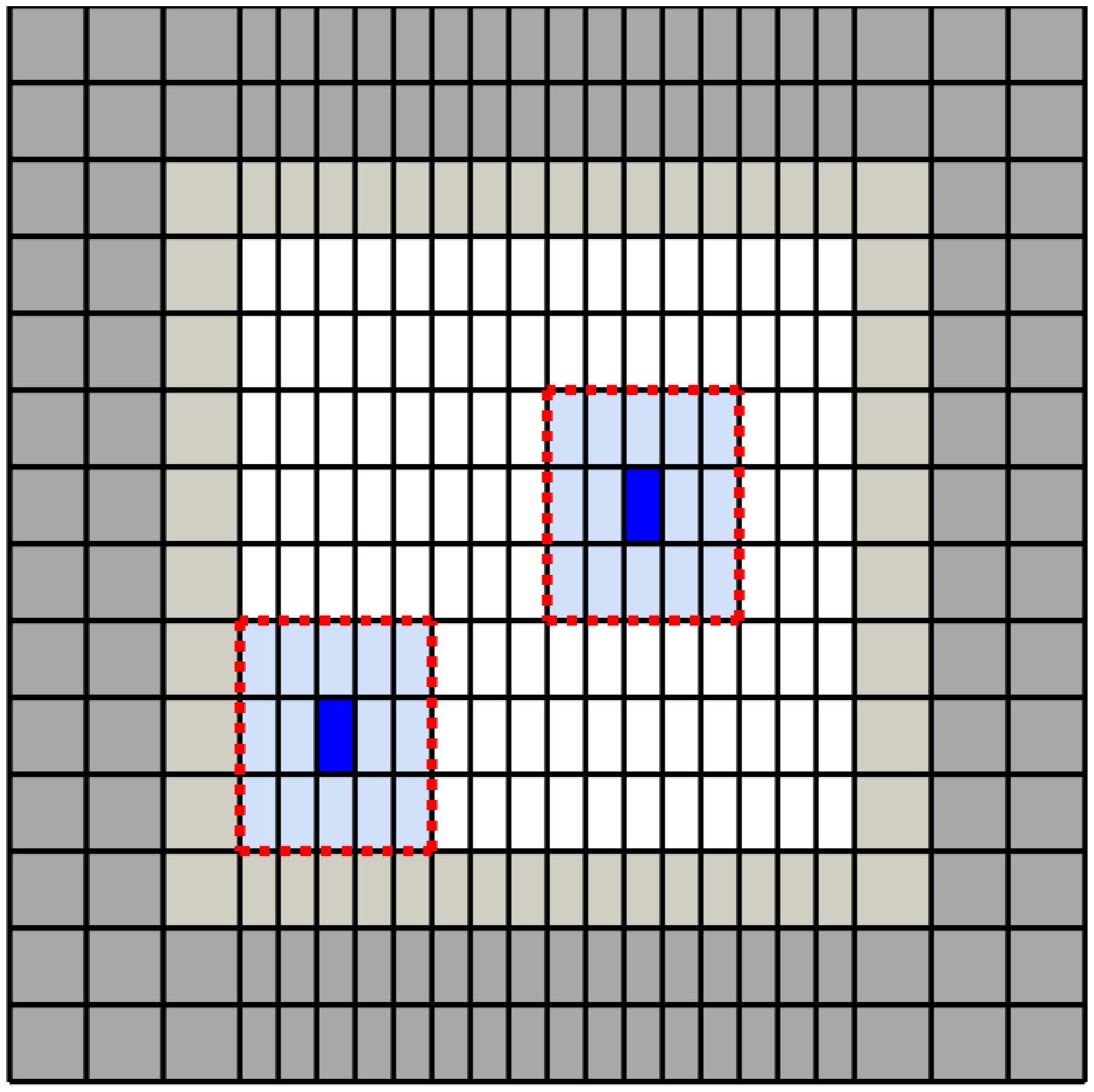}
\end{center}
\caption{\new{An initial T-mesh $\widehat\TT_0^\ex$ \new{in 2D} (left) with $(p_1,p_2)=(3,3)$ and its first uniform refinement (right) are depicted.
The sets $\widehat\Omega$, $\widehat\Omega^\act$, and $\widehat\Omega^\ex$ are highlighted in white, light gray, and dark gray, respectively.
For each of the four blue elements, the corresponding neighbors  are shown in light blue.
According to definition~\eqref{eq:neighbors}, the neighbors are all elements in $\overline{\widehat\Omega}$ with non-empty intersection with the rectangles indicated in red.}}
\label{fig:Tmesh_uni}
\end{figure}

 \begin{figure}[t!] 
\begin{center}
\includegraphics[width=0.25\textwidth,clip=true]{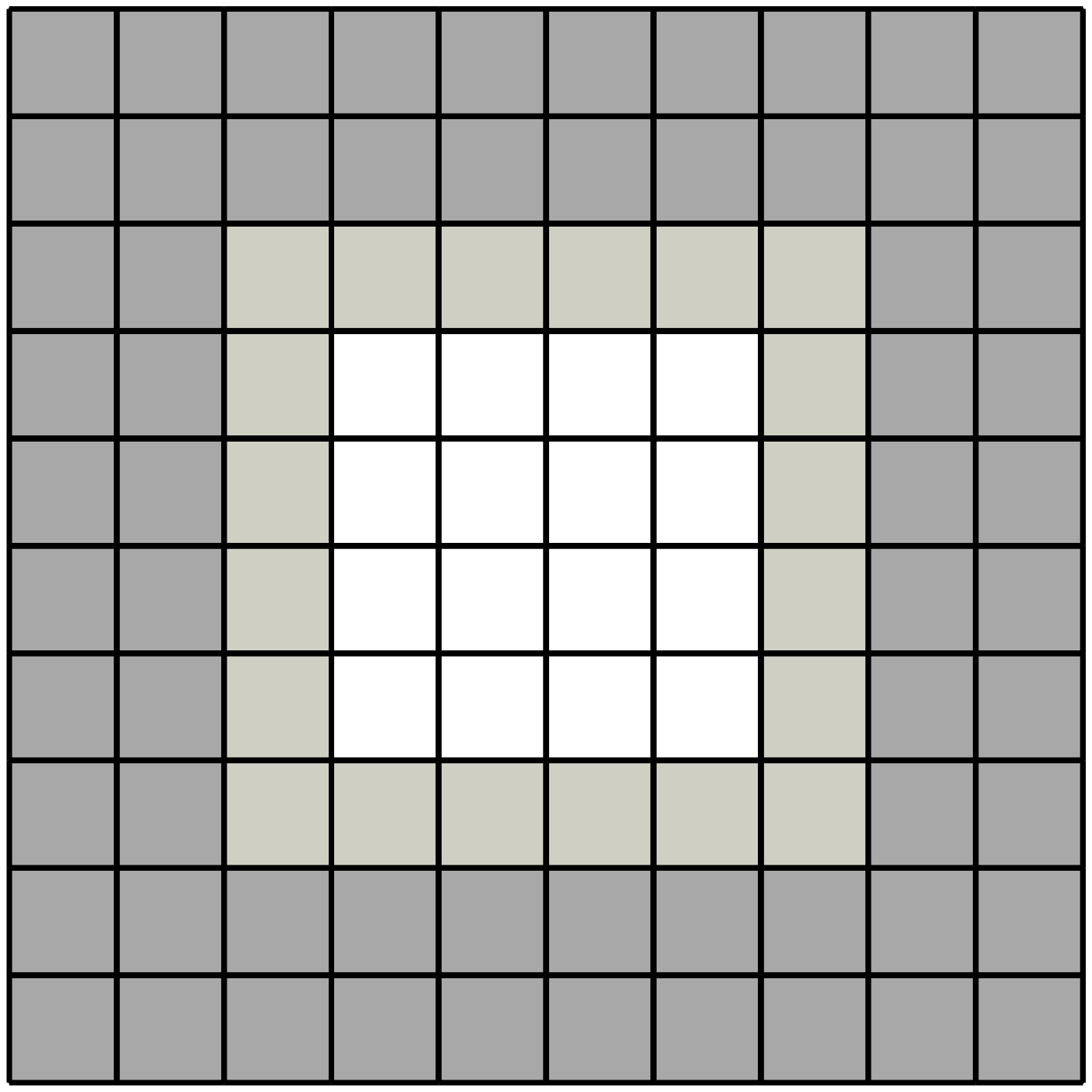}
\quad
\includegraphics[width=0.25\textwidth,clip=true]{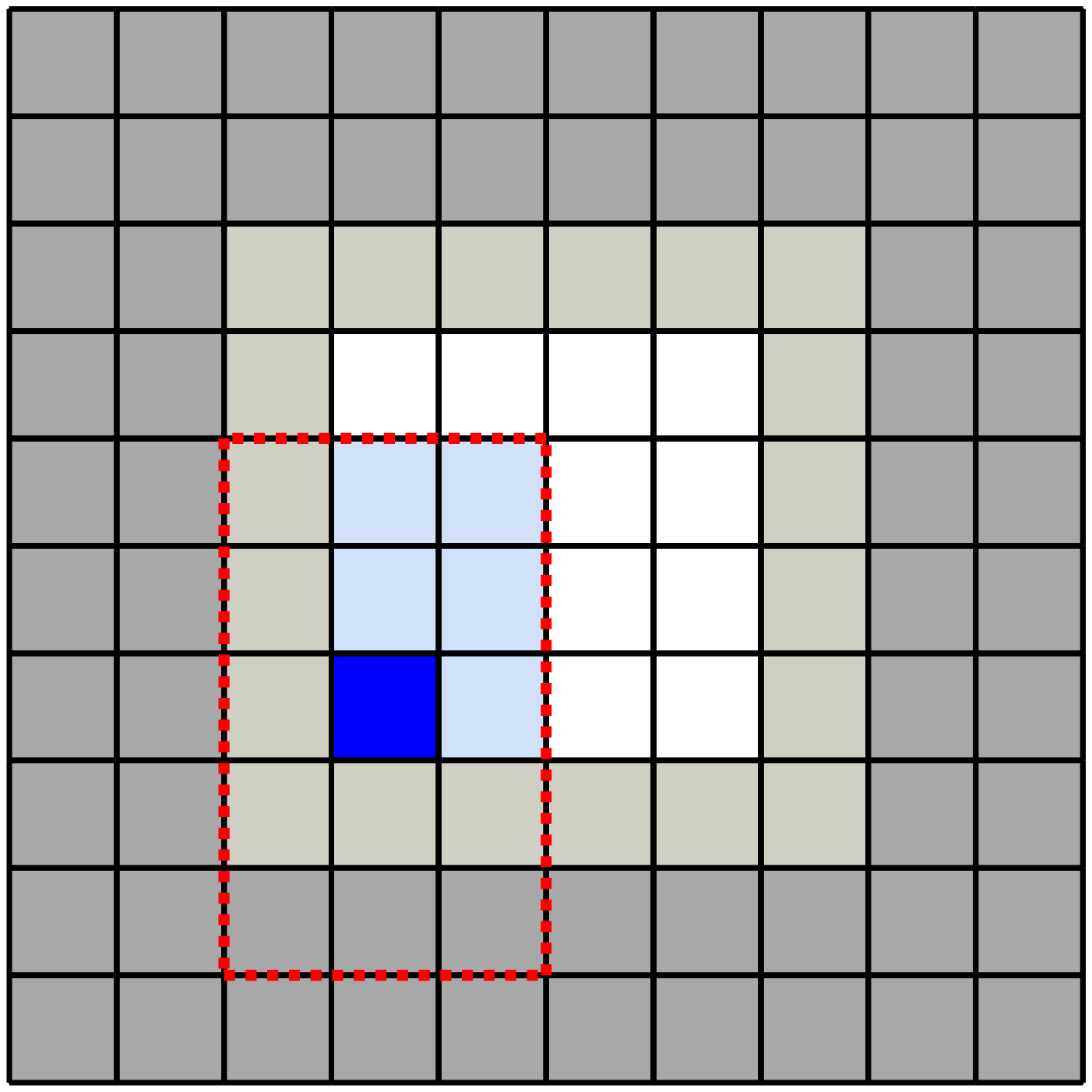}
\quad
\includegraphics[width=0.25\textwidth,clip=true]{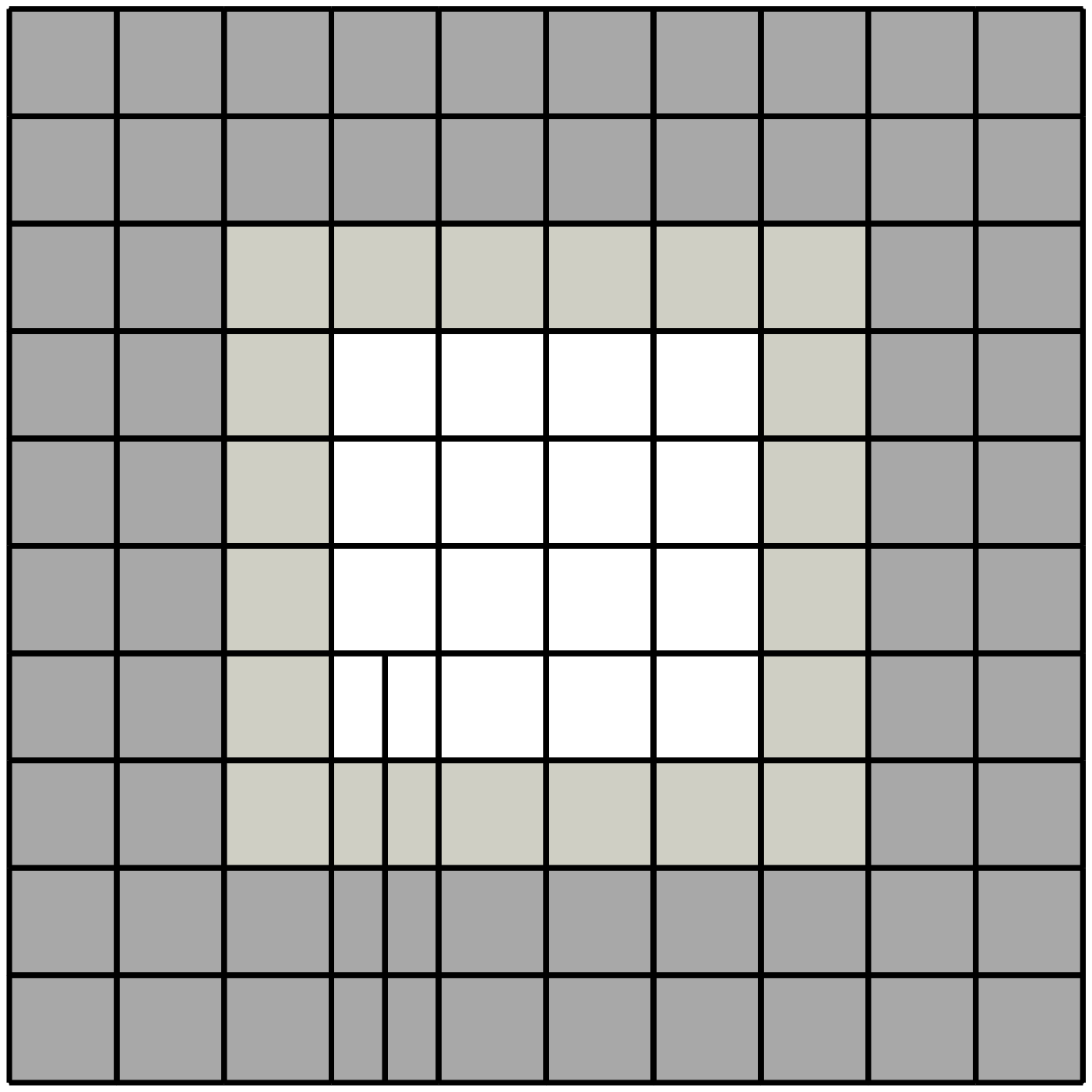}
\\ \vspace{2.5mm}
\includegraphics[width=0.25\textwidth,clip=true]{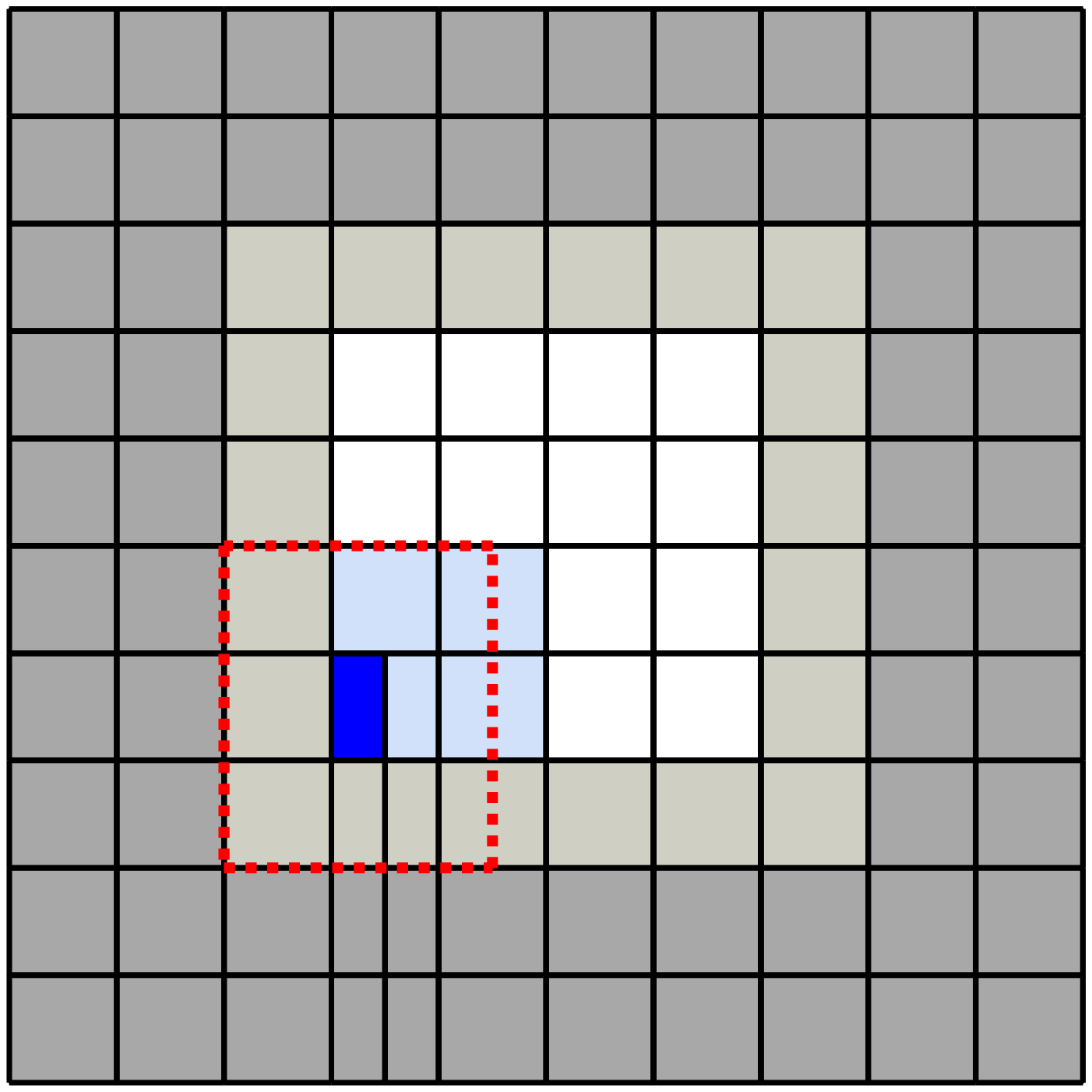}
\quad
\includegraphics[width=0.25\textwidth,clip=true]{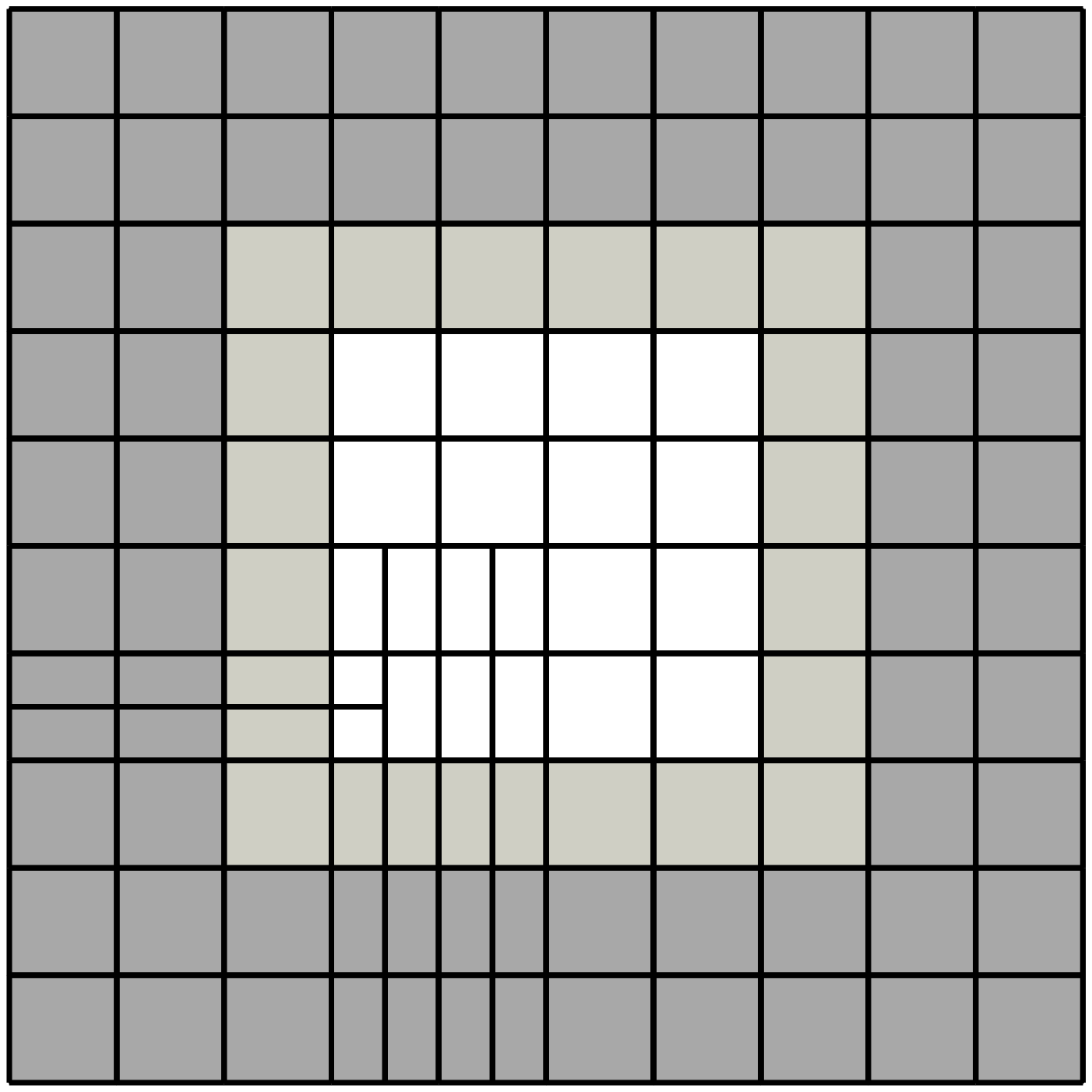}
\quad
\includegraphics[width=0.25\textwidth,clip=true]{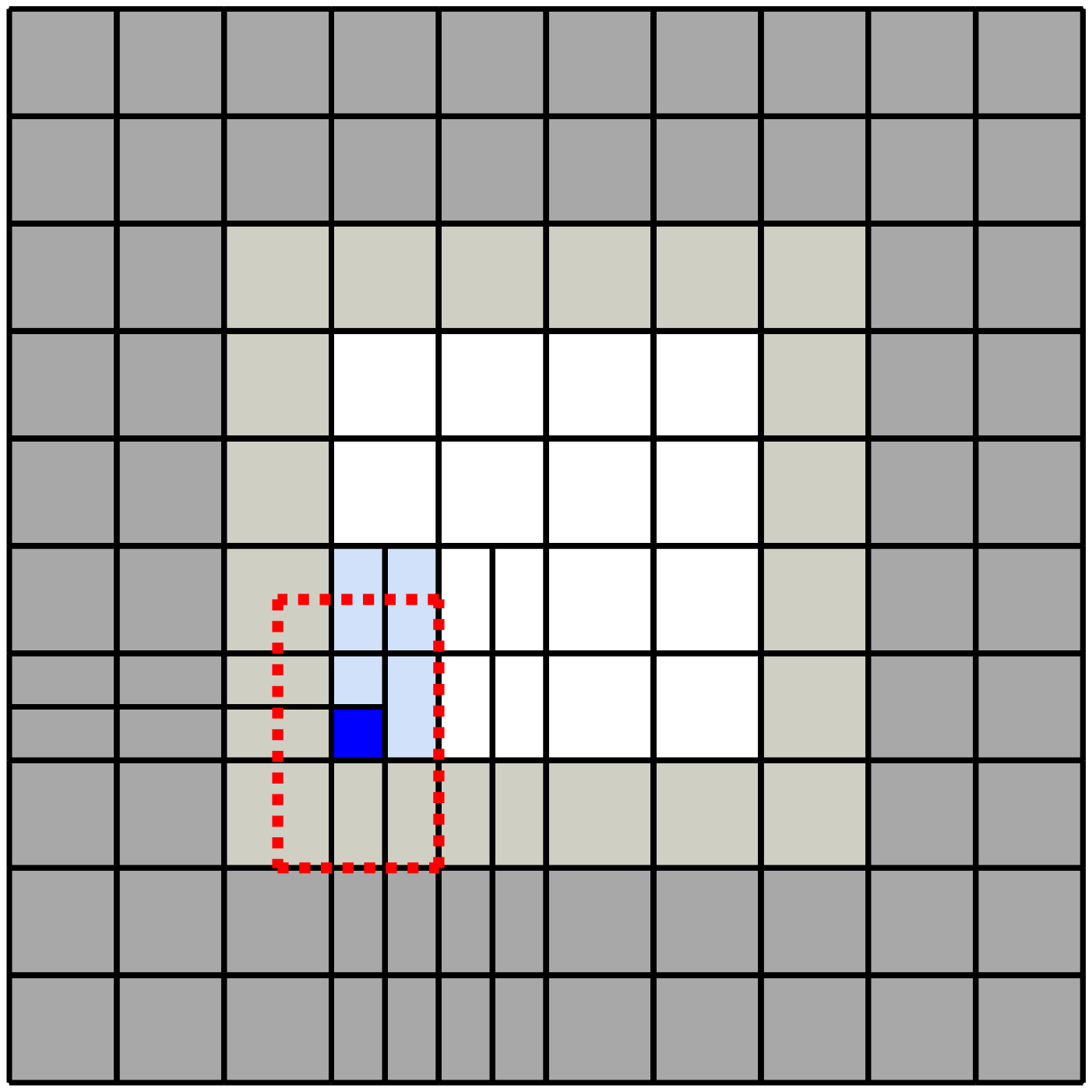}
\\ \vspace{2.5mm}
\includegraphics[width=0.25\textwidth,clip=true]{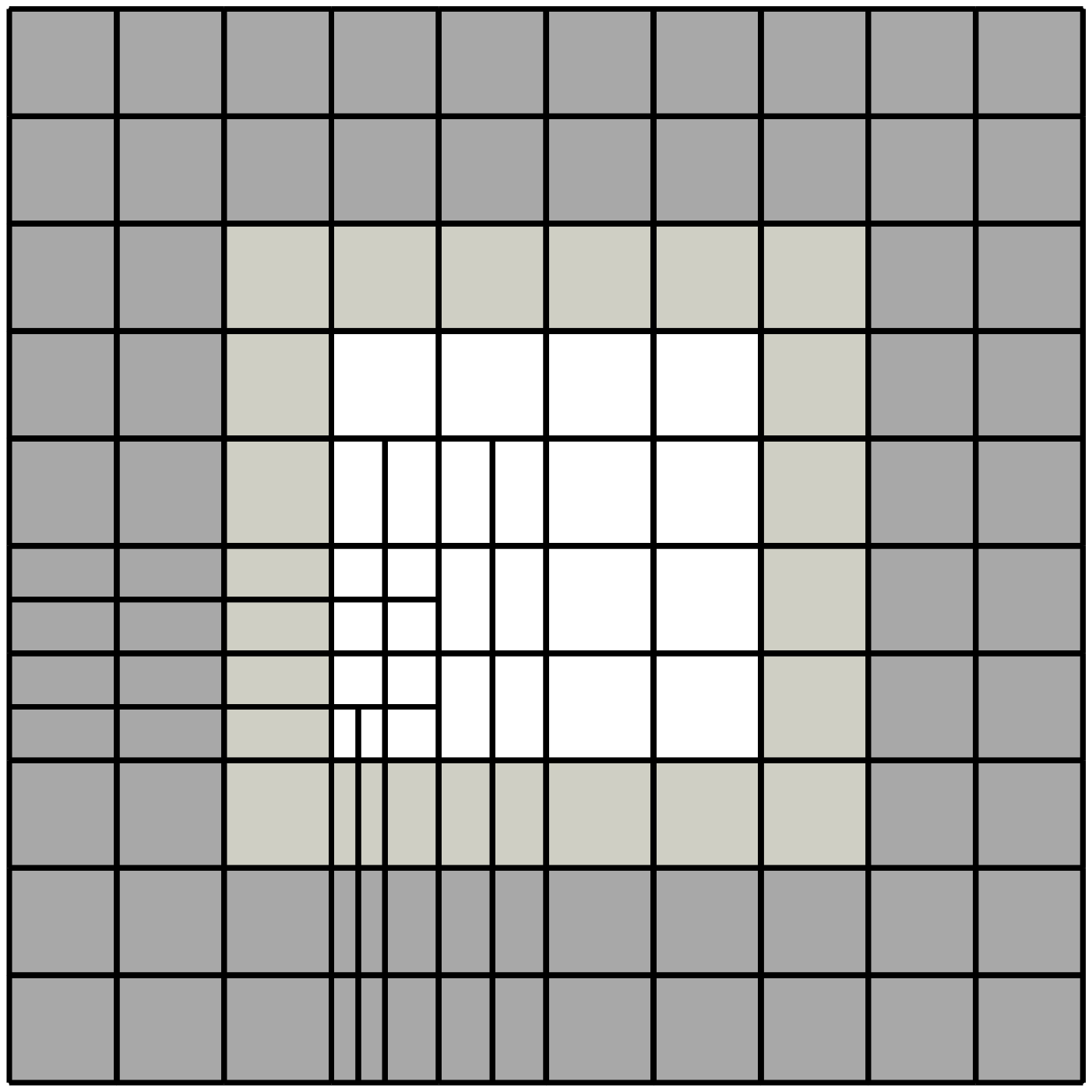}
\quad
\includegraphics[width=0.25\textwidth,clip=true]{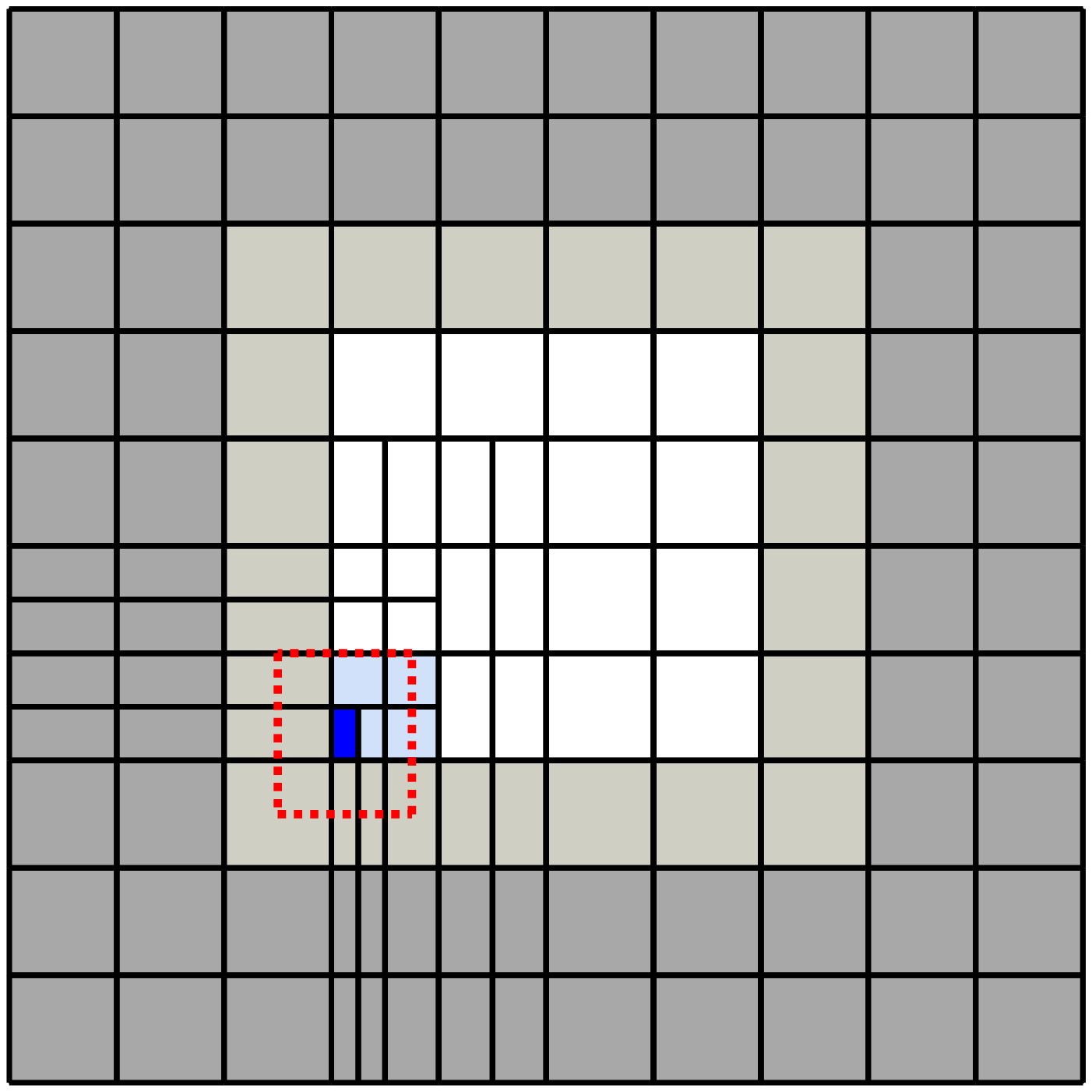}
\quad
\includegraphics[width=0.25\textwidth,clip=true]{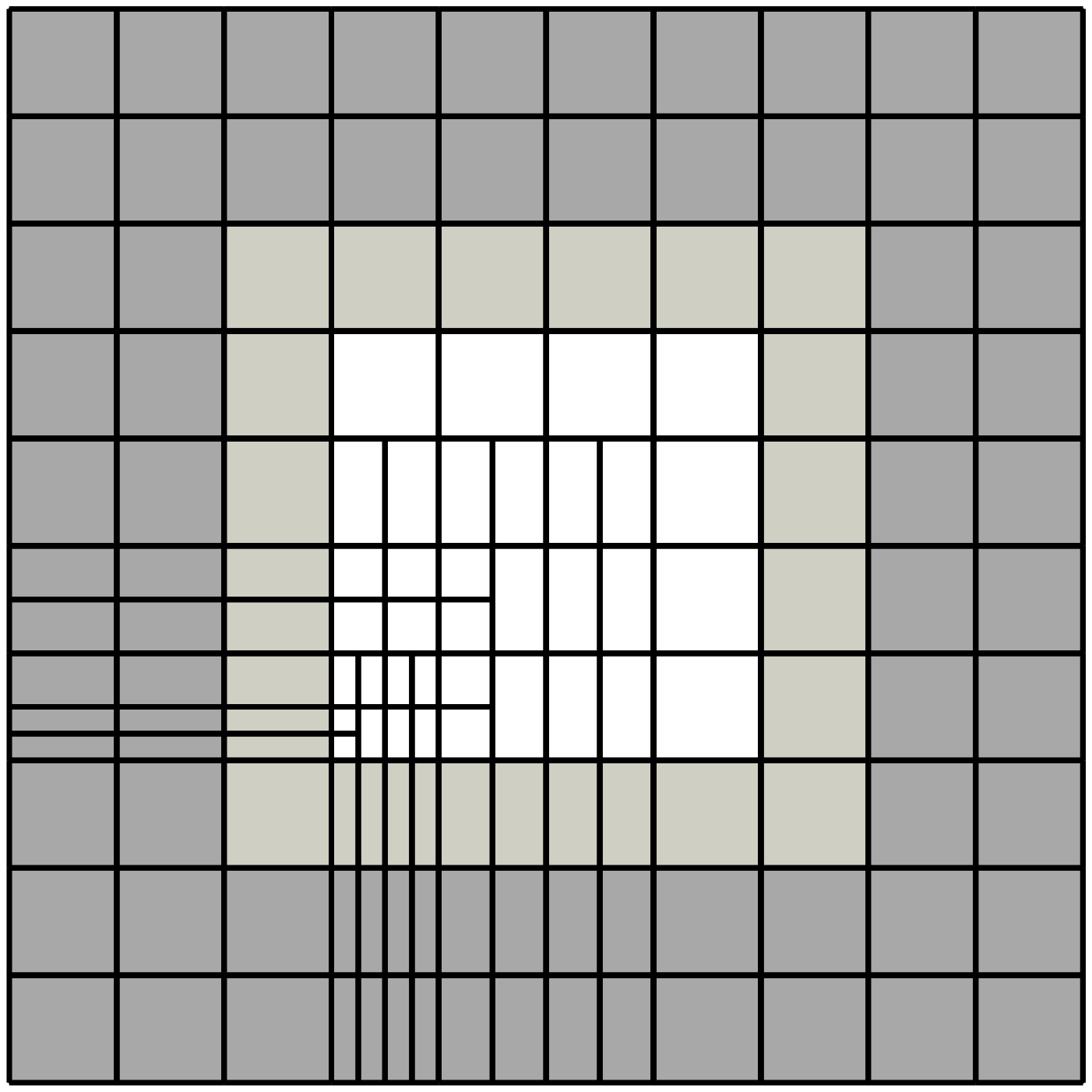}
\\ \vspace{2.5mm}
\includegraphics[width=0.25\textwidth,clip=true]{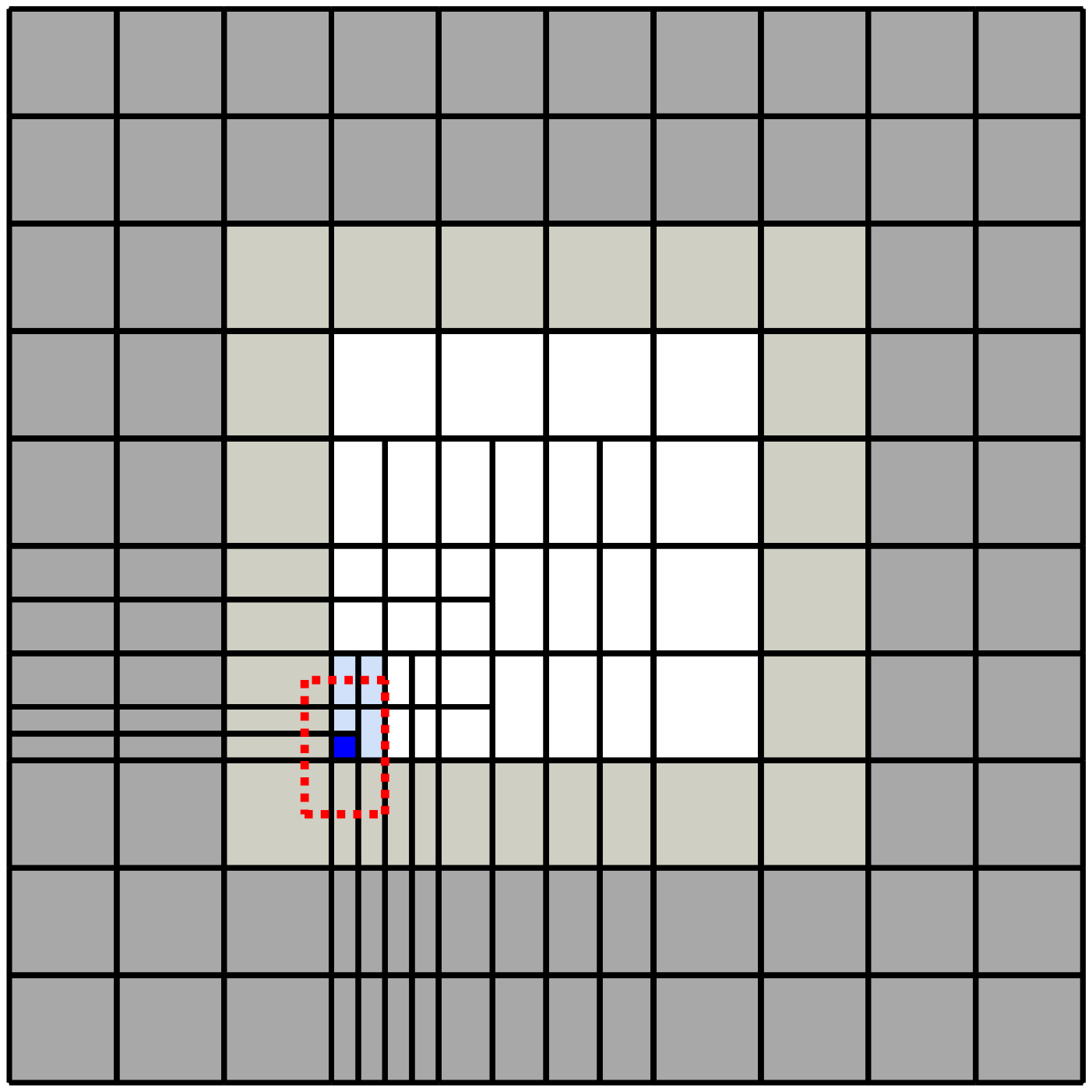}
\quad
\includegraphics[width=0.25\textwidth,clip=true]{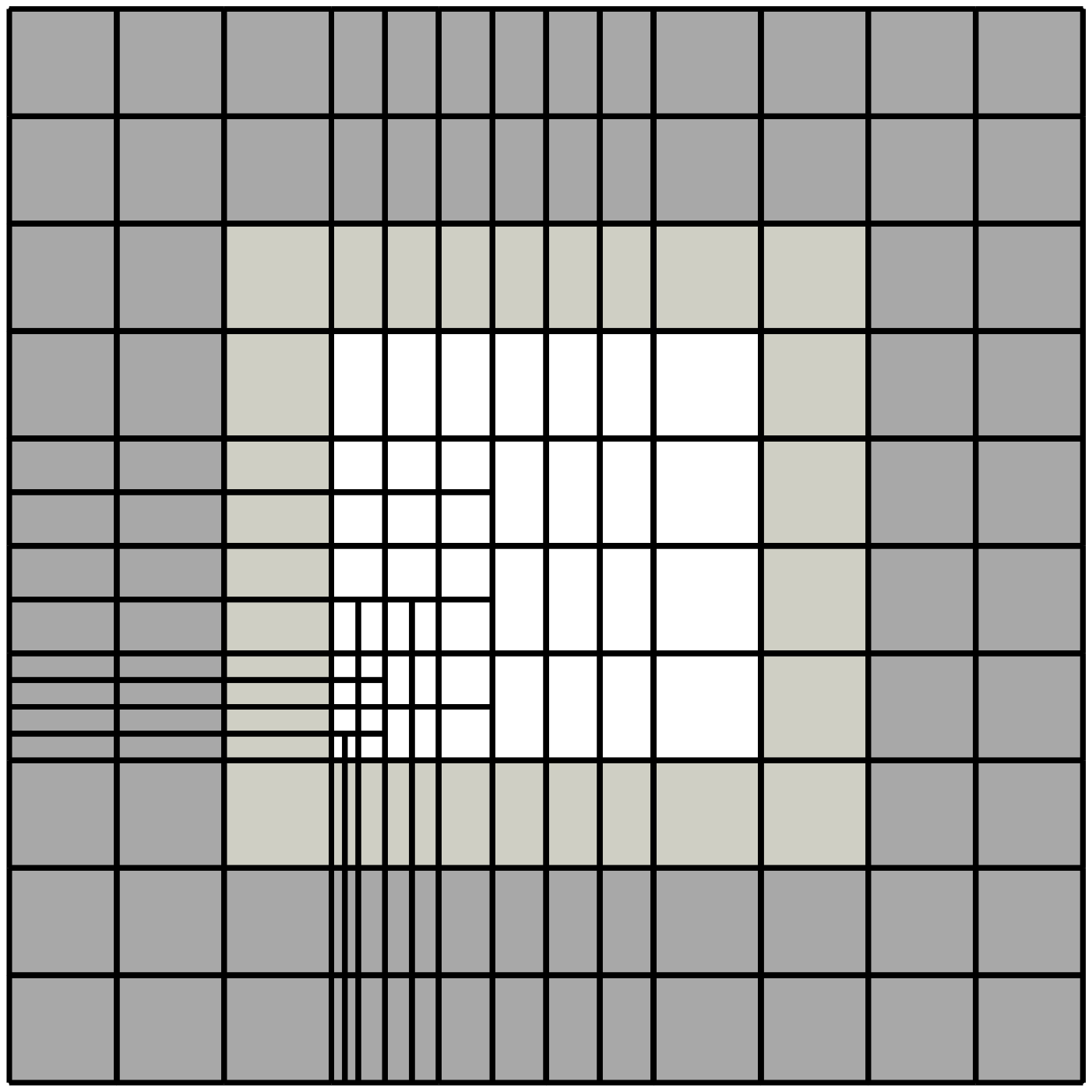}
\quad
\includegraphics[width=0.25\textwidth,clip=true]{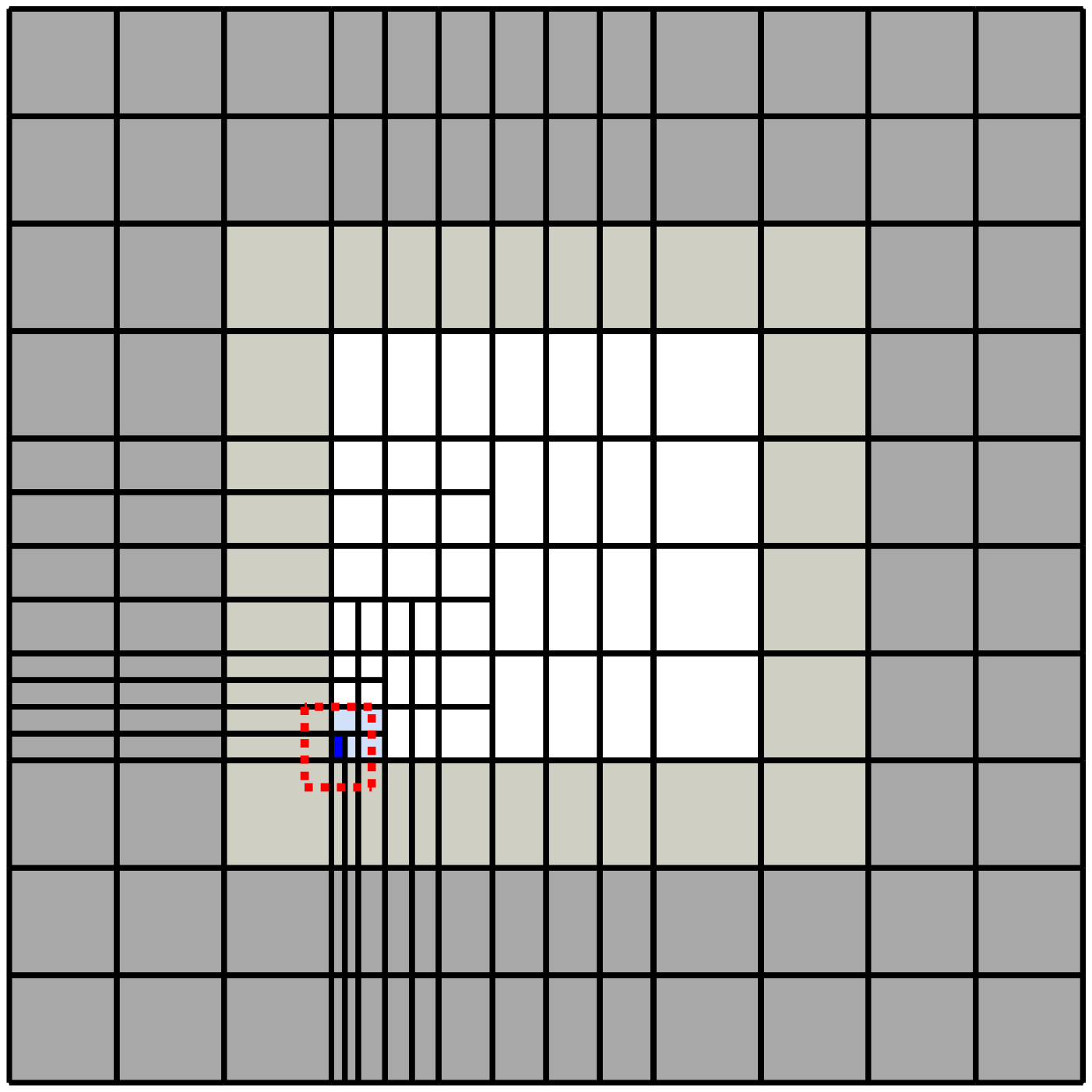}
\end{center}
\caption{\new{An initial T-mesh $\widehat\TT_0^\ex$ in 2D (top left) with $(p_1,p_2)=(3,3)$ and its first five refinements towards the lower left corner of $\widehat\Omega$ are depicted.
The sets $\widehat\Omega$, $\widehat\Omega^\act$, and $\widehat\Omega^\ex$ are highlighted in white, light gray, and dark gray, respectively.
Every second picture shows the element (in blue) that is marked to obtain the next mesh. 
Its neighbors are shown in light blue.
According to definition~\eqref{eq:neighbors}, the neighbors are all elements in $\overline{\widehat\Omega}$ with non-empty intersection with the rectangles indicated in red.}}
\label{fig:Tmeshes}
\end{figure}

 \begin{algorithm}\label{alg:refinement}
{\it\textbf{Input:} T-mesh $\widehat\TT_\bullet$, marked elements $\widehat\MM_\bullet=:\widehat\MM_\bullet^{(0)}\subseteq\widehat\TT_\bullet$.
\begin{itemize}
\item[\rm(i)]  Iterate the following steps {\rm (a)--(b)} for $j=0,1,2,\dots$ until $\widehat\UU_\bullet^{(j)}=\emptyset$:
\begin{itemize}
\item[\rm (a)]Define $\widehat\UU_\bullet^{(j)}:=\bigcup_{\widehat T\in\widehat\MM_\bullet^{(j)}}\boldsymbol{\rm N}^{\rm bad}(\widehat T)\setminus \widehat\MM_\coarse^{(j)}$.
\item[\rm(b)] 
Define  $\widehat\MM_\bullet^{(j+1)}:=\widehat\MM_\bullet^{(j)}\cup\widehat\UU_\bullet^{(j)}$.
\end{itemize}
\item[\rm(ii)] 
Bisect all $\widehat T\in\widehat \MM_\bullet^{(j)}$ via $\bisect_{\level(\widehat T)+1-\underline{\level(\widehat T)}}$ and obtain a finer T-mesh 
\begin{align}\label{eq:bisection}
\refine(\widehat\TT_\bullet,\widehat\MM_\bullet):=\widehat\TT_\coarse\setminus\widehat \MM_\bullet^{(j)}\cup \bigcup\set{\bisect_{\level(\widehat T)+1-\underline{\level(\widehat T)}}(\widehat T)}{\widehat T\in\widehat \MM_\bullet^{(j)}},
\end{align}
where we recall from \eqref{eq:underline} that $\underline{\level(\widehat T)}=\lfloor\level(\widehat T)/d\rfloor d$.
\end{itemize}
\textbf{Output:} Refined mesh $\refine(\widehat\TT_\bullet,\widehat\MM_\bullet)$.}
\end{algorithm}

\new{
\begin{remark}\label{rem:analysis suitable}
The additional bisection of neighbors (and their neighbors, etc.) of marked elements is required to ensure local quasi-uniformity (see~\eqref{eq:shape regular parameter1}--\eqref{eq:shape regular parameter2} below) and \emph{analysis-suitability} in the sense of \cite{beirao,morgensternT2}  for $d=2,3$, respectively.
For $d=2$, the latter is characterized by the assumption that horizontal \emph{T-junction extensions} do not intersect vertical ones. 
In particular, this yields linear independence of the set of B-splines $\set{\widehat B_{\coarse,z}}{z\in\widehat\NN^\act_\coarse}$; see Section~\ref{section:basis}.
\end{remark}}

For any T-mesh $\widehat\TT_\bullet$, we define $\refine(\widehat\TT_\bullet)$ as the set of all T-meshes $\widehat\TT_\circ$ such that there exist T-meshes $\widehat\TT_{(0)},\dots,\widehat\TT_{(J)}$ and marked elements $\widehat\MM_{(0)},\dots,\widehat\MM_{(J-1)}$ with $\widehat\TT_\circ=\widehat\TT_{(J)}=\refine(\widehat\TT_{(J-1)},\widehat\MM_{(J-1)}),\dots,\widehat\TT_{(1)}=\refine(\widehat\TT_{(0)},\widehat\MM_{(0)})$, and $\widehat\TT_{(0)}=\widehat\TT_\bullet$; \new{see Figure~\ref{fig:Tmeshes} for some refined meshes.}
 Here, we formally allow $J=0$, i.e., $\widehat\TT_\coarse\in\refine(\widehat\TT_\coarse)$.
Finally, we define the set of all \textit{admissible T-meshes} as 
\begin{align}
\widehat\T:=\refine(\widehat\TT_0).
\end{align}
For any admissible $\widehat\TT_\coarse\in\widehat\T$, \cite[\new{remark after Definition~2.4} and Lemma~2.14]{morgensternT1} proves for $d=2$ that 
\begin{align}\label{eq:shape regular parameter1}
|\level(\widehat T)-\level(\widehat T')|\le 1\quad\text{for all } \widehat T,\widehat T'\in \widehat \TT_\coarse\text{ with }\widehat T'\in\boldsymbol{\rm N}_\coarse(\widehat T),
\end{align}
as well as  
\begin{align}\label{eq:shape regular parameter2}
\set{\widehat T'\in\widehat\TT_\coarse}{\widehat T\cap\widehat T'\neq\emptyset}\subseteq \boldsymbol{\rm N}_\coarse(\widehat T) \quad\text{for all }\widehat T\in\widehat\TT_\coarse.
\end{align}
Similarly, \cite[Lemma~3.5]{morgensternT2} proves \eqref{eq:shape regular parameter1}--\eqref{eq:shape regular parameter2} for $d=3$.

\new{\begin{remark}
As any element $\widehat T\in\widehat\TT_\coarse\in\T$ of level $k$ is essentially of size $2^{-k/2}$ if $d=2$ and $2^{-k/3}$ if $d=3$, 
the definition of $\boldsymbol{\rm N}_\coarse(\widehat T)$ and \eqref{eq:shape regular parameter1}--\eqref{eq:shape regular parameter2} yield that the number $\#\boldsymbol{\rm N}_\coarse(\widehat T)$ is uniformly bounded independently of the level.
Moreover, for $d=2$, \cite[remark after Definition~2.4]{morgensternT1} states that whenever $\widehat T$ is bisected (in direction $k+1-\underline k$), the resulting sons $\widehat T_1,\widehat T_2$ in the refined mesh $\widehat\TT_\fine$ satisfy that $\boldsymbol{\rm N}_\fine(\widehat T_i)\setminus\{\widehat T_1,\widehat T_2\}\subseteq\boldsymbol{\rm N}_\coarse(\widehat T)\setminus\{\widehat T\}$.
One elementarily sees that this inclusion also holds for $d=3$.
The latter two properties allow for an efficient implementation of Algorithm~\ref{alg:refinement}, where the neighbors of all elements in the current mesh are stored in a suitable data structure and updated after each bisection.
\end{remark}}

\subsection{Basis of $\widehat\XX_\bullet$}\label{section:basis}
First, we emphasize that for general T-meshes $\widehat\TT_\coarse$ as in Section~\ref{subsec:parameter hsplines}, the set $\set{\widehat B_{\coarse,z}}{z\in\widehat\NN^\act_\coarse}$ is not necessarily a basis of the corresponding T-spline space $\widehat\YY_\coarse$ since it is not necessarily linearly independent; see \cite{counter} for a counter example. 
According to \cite[Proposition~7.4]{variational}, a sufficient criterion for linear independence of a set of B-splines is dual-compatibility:
We say that $\set{\widehat B_{\coarse,z}}{z\in\widehat\NN_\coarse}$ is \textit{dual-compatible} if for all $z,z'\in\widehat\NN_\coarse$ with $|\widehat B_{\coarse,z}\cap \widehat B_{\coarse,z'}|>0$, the corresponding local knot vectors are at least in one direction aligned, i.e., there exists $i\in\{1,\dots,d\}$ such that $\widehat\KK^{\rm loc}_{\coarse,i}(z)$ and $\widehat\KK^{\rm loc}_{\coarse,i}(z')$ are both sub-vectors of one common  sorted vector $\widehat\KK$.

We stress that admissible meshes yield dual-compatible B-splines, where the local knot vectors are even aligned in at least two directions for $d=3$, and thus linearly independent B-splines. 
\new{Indeed, \cite[Theorem~3.6]{morgensternT1} and \cite[Theorem~5.3]{morgensternT2} prove analysis-suitability (see Remark~\ref{rem:analysis suitable}) for $d=2$ and $d=3$, respectively.
According to \cite[Theorem~7.16]{variational} for $d=2$ and \cite[Theorem~6.6]{morgensternT2} for $d=3$, this implies the stated dual-compatibility.}
To be precise, 
\cite{morgensternT2} defines the space of T-splines  differently as the span of $\set{\widehat B_{\coarse,z}}{z\in\widehat\NN_\coarse^\act\cap\overline{\widehat\Omega}}$ and shows that this set is dual-compatible. 
The functions in this set are not only zero on the boundary $\partial\widehat\Omega$, but also some of their derivatives vanish there since the maximal multiplicity in the used local knot vectors is at most $p_i$ in each direction; see, e.g.,  \cite[Section~6]{boor}. 
Nevertheless, the proofs immediately generalize to our standard definition of T-splines. 
The following lemma provides a basis of $\widehat\XX_\coarse$.

\begin{lemma}\label{lem:basis of X}
Let $\widehat\TT_\bullet\in\widehat\T$ be an arbitrary admissible T-mesh in the parameter domain $\widehat\Omega$. 
Then, $\set{\widehat B_{\coarse,z}}{z\in\widehat\NN^\act_\coarse\setminus\partial\widehat\Omega^\act}$ is a basis of 
$\widehat\XX_\bullet$.
\end{lemma}
\begin{proof}
Since we already know that the set $\set{\widehat B_{\coarse,z}}{z\in\widehat\NN^\act_\coarse\setminus\partial\widehat\Omega^\act}$ is linearly independent, we only have to show that it generates $\widehat\XX_\coarse$.

\noindent\textbf{Step~1:}
\new{We recall} that the \new{one-dimensional} B-spline $\widehat B(\cdot|\new{x}_0,\dots,\new{x}_{p+1})$ induced by a sorted knot vector $(\new{x}_0,\dots,\new{x}_{p+1})\in\R^{p+2}$ is positive on the interval $(\new{x}_0,\new{x}_{p+1})$. 
It does not vanish at $\new{x}_0$  if and only if $\new{x}_0=\dots=\new{x}_{p}$, \new{and it does not vanish at $\new{x}_{p+1}$ if and only if} $\new{x}_1=\dots=\new{x}_{p+1}$. 
In particular, for all $z\in\widehat\NN_\coarse^\act$,  this \new{and the tensor-product structure of $\widehat B_{\coarse,z}$} yield that $\widehat B_{\coarse,z}|_{\partial\widehat\Omega}\neq 0$ if and only if $z\new{\in}\partial\widehat\Omega^\act$; \new{see also Figure~\eqref{fig:Tmesh}.}
This shows  that
\begin{align}\label{eq:proof:basis}
{\rm span}\set{\widehat B_{\coarse,z}}{z\in\widehat\NN^\act_\coarse\setminus\partial\widehat\Omega^\act}\subseteq\widehat\XX_\coarse.
\end{align}

\noindent\textbf{Step~2:}
To see the other inclusion, let $\widehat V_\coarse\in\widehat\XX_\coarse$. 
Then, there exists a representation of the form $\widehat V_\coarse=\sum_{z\in\widehat\NN_\coarse^\act}c_z \widehat B_{\coarse,z}$.
Let $\widehat E$ be an arbitrary facet of the boundary $\partial\widehat\Omega$ and $\widehat E^\act$ its extension onto $\partial\widehat\Omega^\act$, i.e., 
\begin{align*}
\widehat E&:=\prod_{j=1}^{i-1}[0,N_j]\times\{\widehat e\}\times\prod_{j=i+1}^d[0,N_j], \\
\widehat E^\act&:=\prod_{j=1}^{i-1}[-(p_j-1)/2,N_j+(p_j-1)/2]\times\{\widehat e^{\,\act}\}\times\prod_{j=i+1}^d[-(p_j-1)/2,N_j+(p_j-1)/2],
\end{align*}
with $i\in\{1,\dots,d\}$, $\widehat e:=0$ and $\widehat e^{\,\act}:=-(p_i-1)/2$, or $\widehat e:=N_i$  and $\widehat e^{\,\act}:=N_i+(p_i-1)/2$. 
 Restricting onto $\widehat E$ and using the argument from Step~1, we derive that 
 \begin{align*}
0= \widehat V_\coarse|_{\widehat E}=\sum_{z\in\widehat\NN_\coarse^\act}c_z \widehat B_{\coarse,z}|_{\widehat E}=\sum_{z\in\widehat\NN_\coarse^\act\cap \widehat E^\act}c_z \widehat B_{\coarse,z}|_{\widehat E}.
 \end{align*}

 For $d=2$, the set $\set{\widehat B_{\coarse,z}|_{\widehat E}}{z\in\widehat\NN_\coarse^\act\cap \widehat E^\act}$ coincides (up to the domain of definition) with the set of $(d-1)$-dimensional B-splines corresponding to the global knot vector $\widehat\KK_{\coarse,i}^{\rm gl}((0,0))$ if $\widehat e=0$ and\ $\widehat\KK_{\coarse,i}^{\rm gl}((N_1,N_2))$ if $\widehat e=N_i$; see, e.g.,  \cite[Section~2]{boor} for a precise definition of the set of B-splines associated to some global knot vector.
 It is well-known that these functions are linearly independent, wherefore we derive that $c_z=0$ for the corresponding coefficients.

 For $d=3$, the set $\set{\widehat B_{\coarse,z}|_{\widehat E}}{z\in\widehat\NN_\coarse^\act\cap \widehat E^\act}$ coincides (up to the domain of definition) with the set of $(d-1)$-dimensional B-splines corresponding to the $(d-1)$-dimensional T-mesh 
 \begin{align}
 \widehat\TT_\coarse^\ex|_{\widehat E^\ex}:=\Big\{\prod_{\substack{j=1\\j\neq i}}^{i-1}[a_j,b_j]:\prod_{j=1}^{d}[a_j,b_j]\in \widehat\TT^\ex_\coarse\wedge a_i=\widehat e\Big\}.
 \end{align}
We have already mentioned that
\cite[Theorem~5.3 and Theorem~6.6]{morgensternT2} 
 shows that the local knot vectors of the B-spline basis of $\widehat\YY_\coarse$ are even aligned in at least two directions.
In particular,  the knot vectors of the B-splines corresponding to the mesh $ \widehat\TT_\coarse^\ex|_{\widehat E^\ex}$ are aligned in at least one direction.
This yields dual-compatibility and thus linear independence of these B-splines, which concludes that $c_z=0$ for the corresponding coefficients.
 Since $\partial\widehat\Omega^\act$ is the union of all its facets and $\widehat E$ was arbitrary, this concludes that $c_z=0$ for all $z\in\widehat\NN_\coarse^\act\cap\partial\widehat\Omega^\act$ and thus the other inclusion in~\eqref{eq:proof:basis}.
\end{proof}

Finally, we study the support of the basis functions of $\widehat\YY_\coarse$ (and thus of $\widehat\XX_\coarse$).
To this end, we define for  $\widehat\TT_\coarse\in\widehat\T$ and $\widehat\omega\subseteq\overline{\widehat{\Omega}}$, the patches of order $k\in\N$ inductively by
\begin{align}\label{eq:patch}
 \pi_\coarse^0(\widehat\omega) := \widehat\omega,
 \quad 
 \pi_\coarse^k(\widehat\omega) := \bigcup\set{\widehat T\in\widehat\TT_\coarse}{ {\widehat T}\cap \pi_\coarse^{k-1}(\widehat\omega)\neq \emptyset}.
\end{align}

\begin{lemma}\label{lem:local support}
Let $\widehat\TT_\coarse\in\widehat\T$, $\widehat T\in\widehat\TT_\coarse$, and $z\in\widehat\NN_\coarse^\act$ with $|\widehat T\cap\supp(\widehat B_{\coarse,z})|>0$.
Then, there exists \new{a uniform} $k_{\rm supp}\in\N_0$ such that $\supp(\widehat B_{\coarse,z})\subseteq \pi_\coarse^{k_{\rm supp}}(\widehat T)$.
Moreover, there exist only $k_{\rm supp}$ nodes $z'\in\widehat\NN_\coarse^\act$ such that $|\widehat T\cap\supp(\widehat B_{\coarse,z'})|>0$. 
The constant $k_{\rm supp}$ depends only on $d$ and $(p_1,\dots,p_d)$.
\end{lemma}
\begin{proof}
We prove the assertion in two steps.

\noindent\textbf{Step~1:}
We prove the first assertion.
Without loss of generality, we can assume that $z\in\overline{\widehat\Omega}$, since otherwise there exists $z''\in\widehat\NN_\coarse^\act\cap\overline{\widehat\Omega}$ \new{(namely the projection of $z$ into $\overline{\widehat\Omega}$)} such that $\supp(\widehat B_{\coarse,z})\subseteq\supp(\widehat B_{\coarse,z''})$ \new{(see~\eqref{eq:local knot vector})} and the assertion for $z''$ yields that $\supp(\widehat B_{\coarse,z})\subseteq\pi_\coarse^{k_{\rm supp}}(\widehat T)$. 
Obviously, the support of $\widehat B_{\coarse,z}$ can be covered by elements in $\widehat\TT_\coarse$, i.e., $\supp(\widehat B_{\coarse,z})\subseteq \bigcup\widehat\TT_\coarse=\overline{\widehat\Omega}$.
We show that $|\widehat T|\simeq|\supp(\widehat B_{\coarse,z})|$.
Let $\widehat T_z\in\widehat\TT_\coarse$ with $z\in\widehat T_z$ and thus $\widehat T_z\subseteq \supp(\widehat B_{\coarse,z})$.
Then, \eqref{eq:shape regular parameter1}--\eqref{eq:shape regular parameter2} and the definition of $\widehat B_{\coarse,z}$ show that 
\begin{align}\label{eq:Tz supp}
|\widehat T_z|\simeq|\supp(\widehat B_{\coarse,z})|.
\end{align}
Now, let $z_{\widehat T}\in \widehat\NN_\coarse^\act$ with $z_{\widehat T}\in\widehat T$  and thus $\widehat T\subseteq \supp(\widehat B_{\coarse,z_{\widehat T}})$. 
Then, we have that $|\supp(\widehat B_{\coarse,z})\cap\supp(\widehat B_{\coarse,z_{\widehat T}})|>0$.
Since $\widehat\TT_\coarse$ yields dual-compatible B-splines, the knot lines of $\widehat B_{\coarse,z}$ and $\widehat B_{\coarse,z_{\widehat T}}$ are aligned in one direction.
\new{Moreover, due to~\eqref{eq:shape regular parameter1}--\eqref{eq:shape regular parameter2}, the  difference between consecutive knot lines is equivalent to $2^{-\level(\widehat T_z)/2}$ and $2^{-\level(\widehat T)/2}$, respectively.}
\new{Thus, we obtain} that $\level(\widehat T_z)\simeq\level(\widehat T)$ and $|\widehat T_z|\simeq|\widehat T|$.
In combination with \eqref{eq:Tz supp}, we derive that $|\widehat T|\simeq|\supp(\widehat B_{\coarse,z})|$. 
Since $\widehat T$ is arbitrary and $\supp(\widehat B_{\coarse,z})$ is connected, this   yields the existence of $k_{\rm supp}'\in\N_0$ with $\supp(\widehat B_{\coarse,z})\subseteq \pi_\coarse^{k_{\rm supp}'}(\widehat T)$. 

\noindent\textbf{Step~2:}
We prove the second assertion. 
First, let $z'\in\widehat\NN_\coarse^\act\cap\overline{\widehat\Omega}$. 
Then, Step~1 gives that $z'\in\supp(\widehat B_{\coarse,z'})\subseteq \pi_\coarse^{k_{\rm supp}}(\widehat T)$.
Therefore, we see that the number of such $z'$ is bounded by $\#(\widehat\NN_\coarse^\act\cap\overline{\widehat\Omega}\cap\pi_\coarse^{k_{\rm supp}}(\widehat T))$.
If $z'\in\widehat\NN_\coarse^\act\setminus\overline{\widehat\Omega}$, there exists $z''\in\widehat\NN_\coarse^\act\cap\overline{\widehat\Omega}$ \new{(namely the projection of $z'$ into $\overline{\widehat\Omega}$)} \new{such that} $\supp(\widehat B_{\coarse,\new{z'}})\subseteq\supp(\widehat B_{\coarse,z''})$ \new{and thus $|\supp(\widehat T\cap \widehat B_{\coarse,z''})|>0$}.
On the other hand, for given $z''\in\widehat\NN_\coarse^\act\cap\overline{\widehat\Omega}$, the number of $z'\in\widehat\NN_\coarse^\act\setminus\overline{\widehat\Omega}$ with $\supp(\widehat B_{\coarse,\new{z'}})\subseteq\supp(\widehat B_{\coarse,z''})$ is uniformly bounded by some constant $C>0$ depending only on $d$ and $(p_1,\dots,p_d)$; \new{see also Figure~\ref{fig:Tmesh}.}
Altogether, we see that the number of $z'\in \widehat\NN_\coarse^\act$ with $|\supp(\widehat B_{\coarse,z'})\cap\widehat T|>0$ is bounded by $(1+C)\#(\widehat\NN_\coarse^\act\cap\overline{\widehat\Omega}\cap\pi_\coarse^{k_{\rm supp}}(\widehat T))$.
Due to \eqref{eq:shape regular parameter1}--\eqref{eq:shape regular parameter2}, this term is bounded by some uniform constant $k_{\rm supp}''\in\N_0$.
Finally, we set $k_{\rm supp}:=\max(k_{\rm supp}',k_{\rm supp}'')$.
\end{proof}

\subsection{T-meshes and splines in the physical domain $\bold{\Omega}$}\label{subsec:physical hsplines}
To transform the definitions in the parameter domain $\widehat\Omega$ to the physical domain $\Omega$, we assume as in \cite[Section~3.6]{igafem} that we are given a bi-Lipschitz continuous piecewise $C^2$ parametrization
\begin{align}
\gamma:\overline{\widehat{\Omega}}\to\overline\Omega \quad\text{with}\quad  \gamma\in W^{1,\infty}({\widehat\Omega})\cap C^2(\widehat\TT_0)\quad\text{and}\quad\gamma^{-1}\in W^{1,\infty}(\Omega)\cap C^2(\TT_0),
\end{align}
where $C^2(\widehat\TT_0):=\set{v:\overline\Omega\to\R}{v|_{\widehat T}\in C^2(\widehat T)\text{ for all }\widehat T\in\widehat\TT_0}$ and $C^2(\TT_0):=\set{v:\overline\Omega\to\R}{v|_T\in C^2(T)\text{ for all }T\in\TT_0}$.
Consequently, there exists
$C_{\gamma}>0$ such that for all 
$i,j,k\in\{1,\dots,d\}$
\begin{align}\label{eq:Cgamma}
\begin{split}
\Big\|\frac{\partial}{\partial t_j}\gamma_i\Big\|_{L^\infty(\widehat\Omega)}\le C_\gamma,\quad \Big\|\frac{\partial}{\partial x_j}(\gamma^{-1})_i\Big\|_{L^\infty(\Omega)}\le C_\gamma,\\
\Big\|\frac{\partial^2}{\partial t_j\partial t_k }\gamma_i\Big\|_{L^\infty(\widehat\Omega)}\le C_\gamma,\quad \Big\|\frac{\partial^2}{\partial x_j\partial x_k }(\gamma^{-1})_i\Big\|_{L^\infty(\Omega)}\le C_\gamma,
\end{split}
\end{align}
where $\gamma_i$ resp.~ $(\gamma^{-1})_i$ denotes the $i$-th component of $\gamma$ resp.~ $\gamma^{-1}$ and any second derivative is meant $\TT_0$-elementwise.
All previous definitions can now also be made in the physical domain, just by pulling them from the parameter domain via the diffeomorphism $\gamma$.
For these definitions, we drop the symbol~$\widehat\cdot$.
 Given $\widehat\TT_\bullet\in\widehat\T$,  the corresponding mesh in the physical domain reads  $\TT_\bullet:=\set{\gamma(\widehat T)}{\widehat T\in\widehat\TT_\bullet}$.
In particular, we have that $\TT_0=\set{\gamma(\widehat T)}{\widehat T\in\widehat\TT_0}$.
Moreover, let  $\T:=\set{\TT_\bullet}{\widehat\TT_\bullet\in\widehat\T}$ be the  set of  admissible meshes in the physical domain.
If now $\MM_\bullet\subseteq\TT_\bullet$ with $\TT_\bullet\in\T$, we abbreviate $\widehat\MM_\bullet:=\set{\gamma^{-1}(T)}{T\in\MM_\bullet}$ and define $\refine(\TT_\bullet,\MM_\bullet):=\set{\gamma(\widehat T)}{\widehat T \in\refine(\widehat \TT_\bullet,\widehat\MM_\bullet)}$.
For $\TT_\bullet\in\T$, let  $\YY_\bullet:=\set{\widehat V_\bullet\circ\gamma^{-1}}{\widehat V_\bullet\in\widehat\YY_\bullet}$   be the corresponding space of T-splines, and $\XX_\bullet:=\set{\widehat V_\bullet\circ\gamma^{-1}}{\widehat V_\bullet\in\widehat\XX_\bullet}$   the corresponding space of T-splines which vanish on the boundary.
By  regularity of $\gamma$, we especially obtain that
\begin{align}
\XX_\bullet\subset \set{v\in H_0^1(\Omega)}{v|_T\in  H^2(T)\text{ for all }T\in\TT_\bullet}.
\end{align}
Let $U_\coarse\in\XX_\coarse$ be the corresponding Galerkin approximation to the solution $u\in H^1_0(\Omega)$, i.e.,
\begin{align}
 \edual{U_\coarse}{V_\coarse} = \int_\Omega fV_\coarse-\ff\cdot\nabla V_\coarse\,dx
 \quad\text{for all }V_\coarse\in\XX_\coarse.
\end{align}
We note the Galerkin orthogonality
\begin{align}\label{eq:galerkin}
 \edual{u-U_\coarse}{V_\coarse} = 0
 \quad\text{for all }V_\coarse\in\XX_\coarse
\end{align}
as well as the resulting C\'ea-type quasi-optimality
\begin{align}\label{eq:cea}
 \norm{u-U_\coarse}{H^1(\Omega)}
 \le C_{\text{C\'ea}}\min_{V_\coarse\in\XX_\coarse}\norm{u-V_\coarse}{H^1(\Omega)} 
 \text{ with }
 C_{\text{C\'ea}} := \textstyle\frac{\norm{\AA}{L^\infty(\Omega)}\!+\!\norm{\bb}{L^\infty(\Omega)}\!+\!\norm{c}{L^\infty(\Omega)}}{\Cell}.
\end{align}%

\subsection{Error estimator}\label{subsec:estimator}
Let $\TT_\coarse\in\T$ and $T_1\in\TT_\coarse$.
For almost every $x\in\partial T_1\cap\Omega$, there exists a unique element $T_2\in\TT_\coarse$ with $x\in T_1\cap T_2$.
We denote the corresponding outer normal vectors by $\nu_1$ resp. $\nu_2$ and define the normal jump as 
\begin{align}
 [(\AA \nabla U_\coarse+\ff)\cdot\nu](x)
 = (\AA  \nabla U_\coarse+\ff)|_{T_1}(x)\cdot \nu_1(x)+(\AA  \nabla U_\coarse+\ff)|_{T'}(x)\cdot \nu_2(x).
\end{align}
With this definition, we employ the residual {\sl a~posteriori} error estimator
\begin{subequations}\label{eq:eta}
\begin{align}
 \eta_\coarse := \eta_\coarse(\TT_\coarse)
 \quad\text{with}\quad 
 \eta_\coarse({\mathcal{S}}_\coarse)^2:=\sum_{T\in{\mathcal{S}}_\coarse} \eta_\coarse(T)^2
 \text{ for all }{\mathcal{S}}_\coarse\subseteq\TT_\coarse,
\end{align}
where, for all $T\in\TT_\coarse$, the local refinement indicators read
\begin{align}
\begin{split}
\eta_\coarse(T)^2&:=|T|^{2/d} \norm{f+\div(\AA\nabla U_\coarse+\ff)-\bb\cdot\nabla U_\coarse-c U_\coarse}{L^2(T)}^2\\
&\quad+|T|^{1/d}\norm{[(\AA \nabla U_\coarse+\ff)\cdot \nu]}{L^2(\partial T\cap \Omega)}^2.
\end{split}
\end{align}
\end{subequations}
We refer, e.g., to the monographs \cite{ainsworth-oden,verfuerth} for the analysis of the residual {\sl a~posteriori} error estimator~\eqref{eq:eta} in the frame of standard FEM with piecewise polynomials of fixed order.

\begin{remark}\label{rem:C11}
If $\XX_\bullet\subset C^1(\Omega)$, then the jump contributions in~\eqref{eq:eta} vanish and $\eta_\coarse(T)$ consists only of the volume residual. 
\end{remark}%

\subsection{Adaptive algorithm}
We consider the common formulation of an adaptive mesh-refining algorithm; see, e.g., Algorithm~2.2 of \cite{axioms}.

\begin{algorithm}
\label{the algorithm}
{\it \textbf{Input:} 
Adaptivity parameter $0<\theta\le1$ and marking constant $\Cmin\ge 1$.\\
\textbf{Loop:} For each $\ell=0,1,2,\dots$, iterate the following steps~{\rm(i)}--{\rm(iv)}:
\begin{itemize}
\item[\rm(i)] Compute Galerkin approximation $U_\ell\in\XX_\ell$.
\item[\rm(ii)] Compute refinement indicators $\eta_\ell({T})$
for all elements ${T}\in\TT_\ell$.
\item[\rm(iii)] Determine a set of marked elements $\MM_\ell\subseteq\TT_\ell$ with  $ \theta\,\eta_\ell^2 \le \eta_\ell(\MM_\ell)^2$ which has up to the multiplicative constant $\Cmin$  minimal cardinality.
\item[\rm(iv)] Generate refined mesh $\TT_{\ell+1}:=\refine(\TT_\ell,\MM_\ell)$. 
\end{itemize}
\textbf{Output:} Sequence of successively refined meshes $\TT_\ell$ and corresponding Galerkin approximations $U_\ell$ with error estimators $\eta_\ell$ for all $\ell \in \N_0$.}
\end{algorithm}

\begin{remark}
\new{For the sake of simplicity}, we assume that $U_\ell$ is computed exactly. 
\new{However, according to \cite[Section~7]{axioms}, the forthcoming optimality analysis remains valid if $U_\ell$ is replaced by an approximation $\widetilde U_\ell\in\XX_\ell$ such that 
\begin{align}\label{eq:inexact solve}
\norm{U_\ell-\widetilde U_\ell}{\mathcal{L}}\le \vartheta \eta_\ell(\widetilde U_\ell)
\end{align}
 with the energy norm $\norm{\cdot}{\mathcal{L}}^2:=\edual{\cdot}{\cdot}$, the error estimator $\eta_\ell(\widetilde U_\ell)$ defined analogously as in \eqref{eq:eta}, and a sufficiently small but fixed adaptivity parameter $0<\vartheta<1$.
In practice, \eqref{eq:inexact solve} can be efficiently realized  if one preconditions  the arising linear system appropriately and then solves it iteratively; 
see \cite{fhps19} in case of the boundary element method.
Assuming $\mathcal{L}$ to be symmetric, i.e., $\bb=0$, one employs PCG-iterations~\cite{gl12} starting from $\widetilde U_\ell^0:=\widetilde U_{\ell-1}$ with $\widetilde U_{-1}:=0$ until
\begin{align}
\norm{\widetilde U_\ell^j - \widetilde U_\ell^{j-1}}{\mathcal{L}}\le \vartheta' \eta_\ell(\widetilde U_\ell^j),
\end{align}
and $\widetilde U_\ell$ is defined as the final PCG-iterate $\widetilde U_\ell^j$.}
For analysis-suitable T-splines \new{and the Poisson model problem, appropriate} preconditioners have recently been developed \new{in \cite{cv20}}.
\new{At least for the Poisson model problem, this gives rise to an extended version of  Algorithm~\ref{the algorithm} which does not only converge at optimal rate with respect to the number of mesh elements, but also with respect to the overall computational cost; see \cite{gps19+} for a recent development.}
\end{remark}

\subsection{Data oscillations}
We fix  polynomial orders $(q_1,\dots, q_d)$ and define for $\TT_\coarse \in\T$ the space of transformed polynomials 
\begin{align}\label{eq:polynomials}
\mathcal{P}(\Omega):=\set{\widehat V\circ\gamma}{\widehat V \text{ is a tensor-polynomial of order }(q_1,\dots, q_d)}.
\end{align}

\begin{remark}\label{rem:hot}
In order to obtain higher-order oscillations, the natural choice of the polynomial orders is $q_i\ge2p_i-1$ for $i\in\{1\,\dots,d\}$; see, e.g.,  \cite[Section~3.1]{nv}.
If $\XX_\bullet\subset C^1(\Omega)$, it is sufficient to choose $q_i\ge2p_i-2$; see Remark~\ref{rem:C12}.
\end{remark}

Let $\TT_\coarse\in\T$.
For $T\in\TT_\coarse$, we define the $L^2$-orthogonal projection $P_{\coarse,T}:L^2(T)\to \set{{W}|_{T}}{{W}\in\mathcal{P}(\Omega)}$.
For an interior edge 
$E\in\mathcal{E}_{\coarse,T}:=\set{T\cap T'}{T'\in\TT_\coarse\wedge {\rm dim}(T\cap T')=d-1}$, we define the $L^2$-orthogonal projection  $P_{\coarse,  E}:L^2(E)\to \set{{W}|_{E}}{{W}\in\mathcal{P}(\Omega)}$.
Note that $\bigcup\mathcal{E}_{\coarse,T}=\overline{\partial T\cap \Omega}$.
For $V_\coarse\in\XX_\coarse$, we define  the corresponding oscillations
\begin{subequations}\label{eq:osc}
\begin{align}
 \osc_\coarse(V_\coarse) := \osc_\coarse(V_\coarse,\TT_\coarse)
 \quad\text{with}\quad 
 \osc_\coarse({{V_\coarse}},{\mathcal{S}}_\coarse)^2:=\sum_{T\in{\mathcal{S}}_\coarse} \osc_\coarse({{V_\coarse}}, T)^2
 \text{ for all }{\mathcal{S}}_\coarse\subseteq\TT_\coarse,
\end{align}
where, for all $T\in\TT_\coarse$, the local oscillations read
\begin{align}
\begin{split}
\osc_\coarse({{V_\coarse}},T)^2&:=|T|^{2/d} \norm{(1-P_{\coarse,T})(f+\div(\AA\nabla V_\coarse+\ff)-\bb\cdot\nabla V_\coarse-c V_\coarse)}{L^2(T)}^2\\&\quad+\sum_{E\in\mathcal{E}_{\coarse,T}}|T|^{1/d}\norm{(1-P_{\coarse,E})[(\AA \nabla V_\coarse+\ff)\cdot \nu]}{L^2(E)}^2.
\end{split}
\end{align}
\end{subequations}
We refer, e.g., to \cite{nv} for the analysis of oscillations in the frame of standard FEM with piecewise polynomials of fixed order.
\begin{remark}\label{rem:C12}
If $\XX_\bullet\subset C^1(\Omega)$, then the jump contributions in~\eqref{eq:osc} vanish and $\osc_\coarse(V_\coarse,T)$ consists only of the volume oscillations.  
\end{remark}%

\subsection{Main result}
Let 
\begin{align}
 \T(N):=\set{\TT_\coarse\in\T}{\#\TT_\coarse-\#\TT_0\le N}
 \quad\text{for all }N\in\N_0.
\end{align}
For all $s>0$, define
\begin{align}
 \norm{u}{\mathbb{A}_s}
 := \sup_{N\in\N_0}\min_{\TT_\coarse\in\T(N)}(N+1)^s\,\eta_\coarse\in[0,\infty]
\end{align}
and 
\begin{align}
 \norm{u}{\mathbb{B}_s}
 := \sup_{N\in\N_0}\Big(\min_{\TT_\coarse\in\T(N)}(N+1)^s\,\inf_{V_\coarse\in\XX_\coarse} \big(\norm{u-V_\coarse}{H^1(\Omega)}+\osc_\coarse(V_\coarse)\big)\Big)\in[0,\infty].
\end{align}
By definition, $\norm{u}{\mathbb{A}_s}<\infty$ (\new{or} $\norm{u}{\mathbb{B}_s}<\infty$) implies that the error estimator $\eta_\coarse$ (\new{or} the total error) on the optimal meshes $\TT_\coarse$ decays at least with rate $\OO\big((\#\TT_\coarse)^{-s}\big)$. The following main theorem states that Algorithm~\ref{the algorithm} reaches each possible rate $s>0$.
The proof  builds upon the analysis of \cite{igafem} and is given in Section \ref{sec:proof}.
Generalizations are found in Section~\ref{sec:generalizations}.

\begin{theorem}\label{thm:main}
The following four assertions {\rm (i)--\rm(iv)} \new{hold}:
\begin{itemize}
\item[\rm (i)]
The residual error estimator~\eqref{eq:eta} satisfies reliability, i.e., there exists a constant $\Crel>0$ such that
\begin{align}\label{eq:reliable}
 \norm{u-U_\coarse}{H^1(\Omega)}+\osc_\bullet\le \Crel\eta_\coarse\quad\text{for all }\TT_\coarse\in\T.
\end{align}
\item[\rm(ii)]

The residual error estimator satisfies efficiency, i.e., there  exists  a  constant $\C{eff}>0$ such  that
\begin{align}\label{eq:efficient}
\C{eff}^{-1}\eta_\coarse\le \inf_{V_\coarse\in\XX_\coarse}\big( \norm{u-V_\coarse}{H^1(\Omega)}+\osc_\coarse(V_\coarse)\big).
\end{align}
\item[\rm(iii)]
For arbitrary $0<\theta\le1$ and $\Cmin\ge1$, there exist constants $\Clin>0$ and $0<\qlin<1$ such that the estimator sequence of Algorithm~\ref{the algorithm} guarantees linear convergence in the sense of
\begin{align}\label{eq:linear}
\eta_{\ell+j}^2\le \Clin\qlin^j\eta_\ell^2\quad\text{for all }j,\ell\in\N_0.
\end{align}
\item[\rm (iv)]
There exists a constant $0<\theta_\opt\le1$ such that for all $0<\theta<\theta_\opt$, all $\Cmin\ge1$, and all $s>0$, there exist constants $\copt,\Copt>0$ such that
\begin{align}\label{eq:optimal}
 \copt\norm{u}{\mathbb{A}_s}
 \le \sup_{\ell\in\N_0}{(\# \TT_\ell-\#\TT_0+1)^{s}}\,{\eta_\ell}
 \le \Copt\norm{u}{\mathbb{A}_s},
\end{align}
i.e., the estimator sequence will decay with each possible rate $s>0$. 
\end{itemize}
\noindent The constants $\C{rel},\C{eff},\C{lin},q_{\rm lin},\theta_{\rm opt},$ and $\Copt$ depend only on  $d$, the coefficients of the differential operator $\LL$, $\diam(\Omega)$, $\C{\gamma}$, and $(p_1,\dots,p_d)$, where $\Clin,\qlin$ depend additionally on $\theta$ and the sequence $(U_\ell)_{\ell\in\N_0}$, and $\Copt$ depends furthermore on $\Cmin$, and $s>0$.
Finally, $\copt$ depends only on $\C{son}$, $\#\TT_0$, and $s$.\hfill$\square$
\end{theorem}

\begin{remark}
In particular, it holds that
\begin{align}\label{eq:equivalence}
\C{eff}^{-1}\norm{u}{\mathbb{A}_s}\le\norm{u}{\mathbb{B}_s}\le \C{rel}\norm{u}{\mathbb{A}_s}\quad\text{for all }s>0.
\end{align}
If one  applies continuous piecewise polynomials of degree $p$ on a triangulation of some polygonal \new{or} polyhedral domain $\Omega$ as ansatz space, \cite{morin} proves that $\norm{u}{\mathbb{B}_{p/d}}<\infty$.
The proof  requires that $u$ allows for a certain decomposition and that the oscillations are of higher order; see Remark~\ref{rem:hot}.
In our case, $\norm{u}{\mathbb{A}_s}\simeq\norm{u}{\mathbb{B}_s}$ depends besides the polynomial degrees  $(p_1,\dots,p_d)$ also on the (piecewise) smoothness of the parametrization $\gamma$.  
In practice, $\gamma$ is usually piecewise $C^\infty$. 
Given this additional regularity of $\gamma$, one might expect that the result of \cite{morin} can be generalized such that 
$\norm{u}{\mathbb{A}_{s}},\norm{u}{\mathbb{B}_{s}}<\infty$ for $s=\min_{i=1,\dots,d} p_i/d$.
However, the proof goes beyond the scope of the present work and is left to future research.
\end{remark}

\begin{remark}
Note that almost minimal cardinality of $\MM_\ell$ in Algorithm \ref{the algorithm} {\rm (iii)} is only required to prove optimal convergence behavior \eqref{eq:optimal}, while linear convergence \eqref{eq:linear} formally allows $\Cmin=\infty$, i.e., it suffices that $\MM_\ell$ satisfies the D\"orfler marking criterion.
We refer to \cite[Section 4.3--4.4]{axioms} for details.
\end{remark}

\begin{remark}
{\rm(a)} 
If the bilinear form $\edual{\cdot}{\cdot}$ is symmetric, $\Clin$, $\qlin$ as well as $\copt$, $\Copt$ are  independent of $(U_\ell)_{\ell\in\N_0}$; see \cite[Remark~4.1]{igafem}.

{\rm(b)} If the bilinear form $\edual{\cdot}{\cdot}$ is non-symmetric, there exists an index $\ell_0\in\N_0$ such that the constants $\Clin$, $\qlin$ as well as $\copt$, $\Copt$ are independent of $(U_\ell)_{\ell\in\N_0}$ if \eqref{eq:linear}--\eqref{eq:optimal} are formulated only for $\ell\ge\ell_0$. We refer to the recent work \cite[Theorem~19]{helmholtz}.
\end{remark}%


\begin{remark}
Let $h_\ell:=\max_{T\in\TT_\ell}|T|^{1/d}$ be the maximal mesh-width. Then, $h_\ell\to0$ as $\ell\to\infty$, ensures that
$\XX_\infty:=\overline{\bigcup_{\ell\in\N_0}\XX_\ell}=H_0^1(\Omega)$; see \cite[Remark~2.7]{igafem} for the elementary proof.
We note that the latter observation allows to follow the ideas of \cite{helmholtz}  to show that the adaptive algorithm yields optimal convergence even if the bilinear form $\dual\cdot\cdot_\LL$ is only elliptic up to some compact perturbation provided that the continuous problem is well-posed. This includes, e.g., adaptive FEM for the Helmhotz equation; see \cite{helmholtz}.
\end{remark}




\section{Proof of Theorem~\ref{thm:main}}\label{sec:proof}

In \cite[Section~2]{igafem}, we have identified  abstract properties of the underlying meshes, the mesh-refinement, the finite element spaces,  and the oscillations  which imply Theorem~\ref{thm:main}; see \cite[Section~4.2--4.3]{diss} for more details.
We mention that \cite{igafem,diss} actually only treat the case $\ff=0$, but the corresponding proofs immediately extend to more general $\ff$ as in Section~\ref{sec:model}.
In the remainder of this section, we recapitulate these properties and verify them for our considered T-spline setting.
For their formulation, we define for  $\TT_\coarse\in\T$ and $\omega\subseteq\overline\Omega$, the patches of order $k\in\N$ inductively by
\begin{align}\label{eq:patch}
 \pi_\coarse^0(\omega) := \omega,
 \quad 
 \pi_\coarse^k(\omega) := \bigcup\set{T\in\TT_\coarse}{ {T}\cap \pi_\coarse^{k-1}(\omega)\neq \emptyset}.
\end{align}
The corresponding set of elements is
\begin{align}
 \Pi_\coarse^k(\omega) := \set{T\in\TT_\coarse}{ {T} \subseteq \pi_\coarse^k(\omega)},
 \quad\text{i.e.,}\quad
 \pi_\coarse^k(\omega) = \bigcup\Pi_\coarse^k(\omega){\text{ for }k>0}.
\end{align}
To abbreviate notation, we let $\pi_\coarse(\omega) := \pi_\coarse^1(\omega)$ and $\Pi_\coarse(\omega) := \Pi_\coarse^1(\omega)$.
For ${\mathcal{S}}_\coarse\subseteq\TT_\coarse$, we define $\pi_\coarse^k({\mathcal{S}}_\coarse):=\pi_\coarse^k(\bigcup{\mathcal{S}}_\coarse)$
and $\Pi_\coarse^k({\mathcal{S}}_\coarse):=\Pi_\coarse^k(\bigcup{\mathcal{S}}_\coarse)$. 

\subsection{Mesh properties}
\label{subsec:M true}
We show that there exist $\C{locuni},\Cpatch,\Ctrace,C_{\rm dual}>0$ such that all meshes $\TT_\coarse\in\T$ satisfy the following four properties~\eqref{M:shape}--\eqref{M:dual}:
\begin{enumerate}
\renewcommand{\theenumi}{M\arabic{enumi}}
\bf\item\normalfont\label{M:shape}
\textbf{Local quasi-uniformity.}
For all $T\in\TT_\coarse$ and all $T'\in\Pi_\coarse(T)$, it holds that $\C{locuni}^{-1}|T'| \le |T| \le \C{locuni}\,|T'|$, i.e., neighboring elements have comparable size.
\bf\item\normalfont\label{M:patch}
\textbf{Bounded element patch.}
For all $T\in\TT_\coarse$, it holds that $\#\Pi_\coarse(T)\le\Cpatch$, 
i.e., the number of elements in a patch is uniformly bounded.
\bf\item\normalfont\label{M:trace}
\textbf{Trace inequality.}
For all $T\in\TT_\coarse$ and all $v\in H^1(\Omega)$, it holds that
$\norm{v}{L^2(\partial T)}^2\le \Ctrace \big(|T|^{-1/d}\norm{v}{L^2(T)}^2+ |T|^{1/d}\norm{\nabla v}{L^2(T)}^2\big).$
\bf\item\normalfont\label{M:dual}
\textbf{Local estimate in dual norm:} For all $T\in\TT_\coarse$ and all $w\in L^2(T)$, it holds that
$|T|^{-1/d}\norm{w}{H^{-1}(T)}\le \C{dual}  \norm{w}{L^2(T)},$
where $\norm{{w}}{H^{-1}(T)}:= \sup\set{\int_T{w}  v\,dx}{v\in H_0^1(T)\wedge \norm{ \nabla v}{L^2(T)}=1}$.
\end{enumerate}

\begin{remark}
In usual applications, where $T\in\TT_\coarse$ have simple shapes, the properties \eqref{M:trace}--\eqref{M:dual} are naturally satisfied; see, e.g., \cite[Section~4.2.1]{diss}.
\end{remark}

To see \eqref{M:shape}--\eqref{M:dual}, 
let $\TT_\coarse\in\T$.
Then, \eqref{eq:shape regular parameter1}--\eqref{eq:shape regular parameter2} imply  local quasi-uniformity \eqref{M:shape} in the parameter domain, which transfers with the help of the regularity~\eqref{eq:Cgamma} of $\gamma$ immediately to the physical domain.
The constant $\C{locuni}$ depends only on  the dimension $d$ and the constant $\C{\gamma}$. 
Moreover, \eqref{eq:shape regular parameter1}--\eqref{eq:shape regular parameter2} yield uniform boundedness of the number of elements in a patch, i.e., \eqref{M:patch}, where $\C{patch}$ depends only on $d$. 

Regularity \eqref{eq:Cgamma} of $\gamma$ shows that it is sufficient to prove \eqref{M:trace} for hyperrectangles $\widehat T$ in the parameter domain.
There, the trace inequality \eqref{M:trace} is well-known; see, e.g., \cite[Proposition~4.2.2]{diss} for a proof for general Lipschitz domains.
The constant $\Ctrace$ depends only on  $d$ and $\C{\gamma}$. 

Finally, \eqref{M:dual} in the parameter domain follows immediately from the Poincar\'e inequality.
By regularity \eqref{eq:Cgamma} of $\gamma$, this property transfers to the physical domain.
The constant $\C{dual}$ depends only on  $d$ and $\C{\gamma}$. 

\subsection{Refinement properties}\label{sec:refinement props}

We show that there exist $\Cson\ge2$ and $0<\qson<1$ such that all meshes $\TT_\coarse\in\T$  satisfy for arbitrary marked elements $\MM_\coarse\subseteq\TT_\coarse$ with corresponding refinement $\TT_\fine:=\refine(\TT_\coarse,\MM_\coarse)$, the following elementary properties~\eqref{R:sons}--\eqref{R:reduction}:
\begin{enumerate}
\renewcommand{\theenumi}{R\arabic{enumi}}
\bf\item\normalfont\label{R:sons}
\textbf{Bounded number of sons.}
It holds that $\#\TT_\fine \le \Cson\,\#\TT_\coarse$, i.e., one step of refinement leads to a bounded increase of elements.
\bf\item\normalfont\label{R:union}
\textbf{Father is union of sons.}
It holds that $T=\bigcup\set{{T'}\in\TT_\fine}{T'\subseteq T}$ for all $T\in\TT_\coarse$, i.e., each element $T$ is the union of its successors.
\bf\item\normalfont\label{R:reduction}
\textbf{Reduction of sons.}
It holds that $|T'| \le \qson\,|T|$ for all $T\in\TT_\coarse$ and all $T'\in\TT_\fine$ with $T'\subsetneqq T$, i.e., successors are uniformly smaller than their father.
\end{enumerate}
By induction and the definition of $\refine(\TT_\coarse)$, one easily sees that \eqref{R:union}--\eqref{R:reduction} remain valid for arbitrary $\TT_\fine\in\refine(\TT_\coarse)$.
In particular, \eqref{R:union}--\eqref{R:reduction} imply that each refined element $T\in\TT_\coarse\setminus\TT_\fine$ is split into at least two sons, wherefore 
\begin{align}\label{eq:R:refine}
\#(\TT_\coarse\setminus\TT_\fine)\le \#\TT_\fine-\#\TT_\coarse\quad\text{for all }\TT_\fine\in\refine(\TT_\coarse).
\end{align}

\begin{remark}
In usual applications, the properties \eqref{R:union}--\eqref{R:reduction} are trivially satisfied.
The same holds for \eqref{R:sons} on rectangular meshes.
However, \eqref{R:sons} is not obvious for standard refinement strategies for simplicial meshes; see \cite{gss14} for 3D newest vertex bisection for tetrahedral meshes.
\end{remark}

Moreover,  the following properties~\eqref{R:closure}--\eqref{R:overlay} hold with a generic constant $\Cclos>0$:
\begin{enumerate}
\renewcommand{\theenumi}{R\arabic{enumi}}
\setcounter{enumi}{3}
\bf\item\normalfont\label{R:closure}
\textbf{Closure estimate.}
If $\MM_\ell\subseteq\TT_\ell$ and $\TT_{\ell+1}=\refine(\TT_\ell,\MM_\ell)$ for all $\ell\in\N_0$, then 
\begin{align*}
\# \TT_L-\#\TT_0\le \Cclos\sum_{\ell=0}^{L-1}\#\MM_\ell\quad\text{for all }L\in\N.
\end{align*}
\bf\item\normalfont\label{R:overlay}
\textbf{Overlay estimate.}
For all $\TT_\coarse,\TT_\star\in\T$, there exists a common refinement $\TT_\fine\in\refine(\TT_\coarse)\cap\refine(\TT_\star)$ such that 
\begin{align*}
\#\TT_\fine \le \#\TT_\coarse + \#\TT_\star - \#\TT_0.
\end{align*}
\end{enumerate}

\subsection*{Verification of  {(\ref{R:sons})--(\ref{R:reduction})}}\label{sec:R123}
 \eqref{R:sons} is trivially satisfied with $\Cson=2$, since each refined element is split into exactly two elements.
Moreover, the union of sons property \eqref{R:union} holds by definition.
The reduction property \eqref{R:reduction} in the parameter domain is trivially satisfied and easily transfers to the physical domain with the help of  the regularity  \eqref{eq:Cgamma} of $\gamma$; see \cite[Section~5.3]{igafem} for details.
The constant  $0<q_{\rm son}<1$  depends only on $d$ and $C_\gamma$. 

\subsection*{Verification of {(\ref{R:closure})}}\label{subsec:R3}
The proof of the closure estimate \eqref{R:closure} is found in \cite[Section~6]{morgensternT1} for $d=2$, and in \cite[Section~7]{morgensternT2} for $d=3$.
The constant $\C{clos}$ depends only on the dimension $d$ and the polynomials orders $(p_1,\dots,p_d)$. 

\subsection*{Verification of {(\ref{R:overlay})}}\label{sec:R5}
The proof of the overlay property \eqref{R:overlay} is found in \cite[Section~5]{morgensternT1} for $d=2$.
For $d=3$, the proof follows along the same lines.


\subsection{Space properties}\label{subsec:ansatz}

We show that there exist constants $\Cinv>0$ and $\kloc,\kproj \in\N_0$ such that the following properties~\eqref{S:inverse}--\eqref{S:local} hold for all $\TT_\coarse\in\T$:
\begin{enumerate}
\renewcommand{\theenumi}{S\arabic{enumi}}
\bf\item\normalfont\label{S:inverse}
\textbf{Inverse estimate.}
For  all $i,j\in\{0,1,2\}$ with $j\le i$, all $V_\coarse\in\XX_\coarse$ and all $T\in\TT_\bullet$, it holds that $|T|^{(i-j)/d} \norm{V_\coarse}{H^i(T)}\le \Cinv \, \norm{V_\coarse}{H^j(T)}$.
\bf\item\normalfont\label{S:nestedness}
\textbf{Refinement yields nestedness.}
For  all $\TT_\fine\in\refine(\TT_\coarse)$, it holds that $\XX_\coarse\subseteq\XX_\fine$.
\bf\item\normalfont\label{S:local}
\textbf{Local domain of definition.}
For all $\TT_\fine\in\refine(\TT_\coarse)$ and all
 $T\in\TT_\coarse\setminus \Pi_\coarse^{k_{\rm loc}}( \TT_\coarse\setminus\TT_\fine)\subseteq\TT_\coarse\cap\TT_\fine$,  it holds  that
$V_\fine|_{\pi_\coarse^{k_{\rm proj}}(T)} \in \set{V_\coarse|_{\pi_\coarse^{k_{\rm proj}}(T)}}{V_\coarse\in\XX_\coarse}$.
\end{enumerate}

Moreover, we show that there exist $\C{sz}>0$ and $\kapp, k_{\rm  grad}\in\N_0$ such that for all $\TT_\coarse\in\T$, there exists a Scott--Zhang-type projector $J_\coarse:H^1_0(\Omega)\to\XX_\coarse$ with the following properties~\eqref{S:proj}--\eqref{S:grad}:
\begin{enumerate}
\renewcommand{\theenumi}{S\arabic{enumi}}
\setcounter{enumi}{3}
\bf\item\normalfont\label{S:proj}
\textbf{Local projection property.}
With $\kproj\in\N_0$ from \eqref{S:local}, 
for all $v\in H^1_0(\Omega)$ and $T\in\TT_\coarse$, it holds that $(J_\coarse v)|_T = v|_T$, if $v|_{\pi_\coarse^{k_\proj}(T)} \in \set{V_\coarse|_{\pi_\coarse^{k_\proj}(T)}}{V_\coarse\in\XX_\coarse}$.
%
\bf\item \normalfont
\textbf{Local $\boldsymbol{L^2}$-approximation property.}
For all $T\in\TT_\coarse$ and all $v\in H_0^1(\Omega)$, it holds that $\norm{(1-J_\coarse)v}{L^2(T)}\le \C{sz} \,|T|^{1/d}\,\norm{v}{H^1(\pi_\coarse^{k_\mathrm{app}}(T))}$.
\label{S:app}
\bf\item\normalfont\label{S:grad}
\textbf{Local $\boldsymbol{H^1}$-stability.}
For all $T\in\TT_\coarse$ and $v\in H_0^1(\Omega)$, it holds that $\norm{\nabla J_\coarse v}{L^2(T)}\le \C{sz} \norm{v}{H^1(\pi_\coarse^{k_\mathrm{grad}}(T))}$.
\end{enumerate}

\subsection*{Verification of {(\ref{S:inverse})}}\label{subsec:E1.1 true}
Let $T\in\TT_\coarse\in\T$.
Let $V_\coarse\in\XX_\coarse$.
Define $\widehat V_\coarse:=V_\coarse\circ\gamma\in\widehat\XX_\coarse\subseteq\widehat\YY_\coarse$ and $\widehat T:=\gamma^{-1}(T)\in\widehat\TT_\coarse$.
Regularity \eqref{eq:Cgamma} of $\gamma$ proves for $i\in\{0,1,2\}$ that
\begin{align}\label{eq:sobolev equivalent}
\norm{V_\bullet}{H^i(T)} 
\simeq\norm{\widehat V_\bullet}{H^i(\widehat T)},
\end{align}
where the hidden constants depend only on $d$ and $C_\gamma$.
Thus, it is sufficient to prove \eqref{S:inverse} in the parameter domain.
In general, $\widehat V_\coarse$ is not a $\widehat \TT_\coarse$-piecewise tensor-polynomial.
However, there is a uniform constant $k\in\N_0$ depending only on $d$ and $(p_1,\dots,p_d)$ such that $\widehat V_\coarse|_{\widehat T'}$ is a tensor-polynomial on any $k$-times refined son $\widehat T'\subseteq \widehat T$ with $\widehat T'\in \widehat\TT_{{\rm uni}(\level(\widehat T)+k)}$:

To see this, we use Lemma~\ref{lem:local support}, which yields that the number of B-splines $\widehat B_{\coarse,z}$ which are needed in the linear combination of $\widehat V_\coarse|_{\widehat T}$, i.e., $\widehat B_{\coarse,z}$  with $|\supp(\widehat B_{\coarse,z})\cap\widehat T|\new{>0}$, is uniformly bounded by $k_{\rm supp}$. 
Moreover, Lemma~\ref{lem:local support} and local quasi-uniformity \eqref{eq:shape regular parameter1}--\eqref{eq:shape regular parameter2} show that $\level(\widehat T'')\simeq\level(\widehat T)$ for all elements $\widehat T''\in\widehat\TT_\coarse$ which satisfy that $|\supp(\widehat B_{\coarse,z})\cap \widehat T''|> 0$ for any of these B-splines.
Since we only allow dyadic bisections, the definition of $\widehat B_{\coarse,z}$ yields the existence of $k\in\N_0$ depending only on $d$ and $(p_1,\dots,p_d)$ such that $\widehat B_{\coarse,z}|_{\widehat T'}$ and thus $\widehat V_\coarse|_{\widehat T'}$ are  tensor-product polynomials for any son $\widehat T'\subseteq \widehat T$ with $\widehat T'\in \widehat\TT_{{\rm uni}(\level(\widehat T)+k)}$. 

In particular, we can apply a standard scaling argument on $\widehat T'$.
Since $\widehat T$ is the union of all such sons and $|\widehat T|\simeq|\widehat T'|$, this yields that 
\begin{align}\label{eq:parameter invest}
|\widehat T|^{(i-j)/d} \norm{\widehat V_\bullet}{H^i(\widehat T)}\lesssim \norm{\widehat V_\bullet}{H^j(\widehat T)},
\end{align}
where  the hidden constant  depends only on $d$ and $(p_1,\dots,p_d)$.
Together, \eqref{eq:sobolev equivalent}--\eqref{eq:parameter invest} conclude the proof of \eqref{S:inverse}, where $C_{\rm inv}$ depends only on $d$, $C_\gamma$,  and $(p_1,\dots,p_d)$.

\subsection*{Verification of {(\ref{S:nestedness})}}\label{subsec:nestedness}
We note that in general, i.e., for arbitrary T-meshes, nestedness of the induced T-splines spaces is not evident; see, e.g., \cite[Section~6]{li2014analysis}.
However, the refinement strategies (Algorithm~\ref{alg:refinement}) from \cite{morgensternT1,morgensternT2} yield nested T-spline spaces.
For $d=2$, this is stated in \cite[Corollary~5.8]{morgensternT1}. 
For $d=3$, this is stated in \cite[Theorem~5.4.12]{diss_morgenstern}.
We already mentioned in Section~\ref{section:basis} that \cite{morgensternT2} (as well as \cite{diss_morgenstern}) define the space of  T-splines  differently as the span of $\set{\widehat B_{\coarse,z}}{z\in\widehat\NN_\coarse^\act\cap\overline{\widehat\Omega}}$.
Nevertheless, the proofs immediately generalize to our standard definition of T-splines, i.e., 
\begin{align}\label{eq:YY nested}
\widehat\YY_\coarse\subseteq\widehat\YY_\fine\quad\text{for all }\widehat\TT_\coarse\in\widehat\T,\widehat\TT_\fine\in\refine(\widehat\TT_\coarse),
\end{align}
which also yields the  inclusion $\XX_\coarse\subseteq\XX_\fine$.

\subsection*{Verification of {(\ref{S:local})}}\label{subsec:E4.2 true}
We show the assertion in the parameter domain.
For arbitrary but fixed $\kproj\in\N_0$ (which will be fixed later in Section~\ref{subsec:E4.1 true} to be $k_{\rm proj}:= k_{\rm supp}$), we  set $k_{\rm loc}:=k_{\rm proj}+k_{\rm supp}$ with $k_{\rm supp}$ from Lemma~\ref{lem:local support}. 
Let $\widehat \TT_\bullet\in\widehat\T$,  $\widehat\TT_\circ\in\refine(\widehat\TT_\bullet)$, and $\widehat V_\circ\in\widehat\XX_\circ$.
We define the patch functions $\pi_\bullet(\cdot)$ and $\Pi_\bullet(\cdot)$ in the parameter domain analogously to the patch functions in the physical domain.
Let $\widehat T\in\widehat\TT_\bullet\setminus\Pi_\bullet^{k_{\rm loc}}(\widehat\TT_\bullet\setminus\widehat\TT_\circ)$. 
Then, one easily shows that 
\begin{align}\label{eq:omegainv}
\Pi_\bullet^{k_{\rm loc}}(\widehat T)\subseteq \widehat\TT_\bullet\cap\widehat\TT_\circ;
\end{align}
see \cite[Section~5.8]{igafem}.
We see that $\widehat\omega=\pi_\fine^{k_{\rm loc}}(\widehat T)$, and,  in particular, also  $\widehat\omega:=$ $\pi_\coarse^{k_{\rm proj}}(\widehat T)=\pi_\fine^{k_{\rm proj}}(\widehat T)$.
According to Lemma~\ref{lem:basis of X}, 
it holds that
\begin{align*}
\set{\widehat V_\coarse|_{\widehat\omega}}{\widehat V_\coarse\in \widehat \XX_\coarse}={\rm span}\set{\widehat B_{\coarse,z}|_{\widehat\omega}}{\big(z\in\widehat\NN^\act_\coarse\setminus \partial\widehat\Omega^\act\big)\wedge\big(|\supp(\widehat B_{\coarse,z})\cap\widehat\omega|>0\big)},
\end{align*}
as well as
\begin{align*}
\set{\widehat V_\fine|_{\widehat\omega}}{\widehat V_\fine\in\widehat\XX_\fine}={\rm span}\set{\widehat B_{\fine,z}|_{\widehat\omega}}{\big(z\in\widehat\NN^\act_\fine \setminus \partial\widehat\Omega^\act\big)\wedge\big(|\supp(\widehat B_{\fine,z})\cap\widehat\omega|>0\big)}.
\end{align*}
We will prove that
\begin{align}\label{eq:local ansatz prove}
\set{\widehat B_{\coarse,z}}{z\in\widehat\NN^\act_\coarse\wedge|\supp(\widehat B_{\coarse,z})\cap\widehat\omega|>0}
=\set{\widehat B_{\fine,z}}{z\in\widehat\NN^\act_\fine\wedge|\supp(\widehat  B_{\fine,z})\cap\widehat\omega|>0},
\end{align}
which will conclude \eqref{S:local}.
To show "$\subseteq$", let $\widehat B_{\coarse,z}$ be an element of the left set.
By Lemma~\ref{lem:local support}, this implies that $\supp(\widehat B_{\coarse,z})\subseteq\pi_\bullet^{k_{\rm loc}}(\widehat T)$.
Together with \eqref{eq:omegainv}, we see that $\supp(\widehat B_{\coarse,z})\subseteq \bigcup(\widehat\TT_\bullet\cap \widehat\TT_\circ)$.
This proves that no element within $\supp(\widehat B_{\coarse,z})$ is changed during refinement.
Thus, the definition of  T-spline basis functions proves that $\widehat B_{\coarse,z}=\widehat B_{\fine,z}$. 
The same argument shows the converse inclusion "$\supseteq$". 
This \new{proves} \eqref{eq:local ansatz prove}, and thus \eqref{S:local} follows.

\subsection*{Verification of {(\ref{S:proj})--(\ref{S:grad})}}\label{subsec:E4.1 true}
Given $\TT_\bullet\in\T$, we introduce a suitable Scott--Zhang-type operator $J_\bullet:H_0^1(\Omega)\to\XX_\bullet$  which satisfies \eqref{S:proj}--\eqref{S:grad}.
To this end, it is sufficient to construct a corresponding operator $\widehat J_\bullet:H_0^1(\widehat\Omega)\to\widehat\XX_\bullet$ in the parameter domain, and to define 
\begin{align}
J_\bullet v:=\big(\widehat J_\bullet (v\circ\gamma)\big)\circ\gamma^{-1}\quad\text{for all }v\in H_0^1(\Omega).
\end{align}
By regularity \eqref{eq:Cgamma} of $\gamma$, the properties \eqref{S:proj}--\eqref{S:grad} immediately transfer from the parameter domain $\widehat\Omega$ to the physical domain $\Omega$.
In Section~\ref{section:basis}, we have already mentioned  that any admissible mesh $\widehat\TT_\coarse\in\widehat\T$ yields dual-compatible B-splines $\set{\widehat B_{\coarse,z}}{z\in\widehat\NN^\act_\coarse}$.
According to \cite[Section~2.1.5]{variational} in combination with \cite[Proposition~7.3]{variational} for $d=2$ and with \cite[Theorem~6.7]{morgensternT2} for $d=3$, this implies for all $z\in\widehat\NN_\coarse$ the existence of a local dual function $\widehat B_{\coarse,z}^*\in L^2(\widehat\Omega)$ with $\supp(\widehat B_{\coarse,z}^*)=\supp(\widehat B_{\coarse,z})$ such that 
\begin{align}\label{eq:duality}
\int_{\widehat\Omega} \widehat B_{\coarse,z}^* \,\widehat B_{\coarse,z'}\,dt=\delta_{z,z'}\quad\text{for all }z'\in\widehat\NN_\coarse^\act,
\end{align}
and 
\begin{align}\label{eq:dual inequality}
\norm{\widehat B_{\coarse,z}^*}{L^2(\widehat \Omega)}\le \prod_{i=1}^d\big({9^{p_i}}(2p_i+3)^d\big) \, {|\supp( \widehat B_{\coarse,z})|^{-1/2}}.
\end{align}
With these dual functions, it is easy to define a suitable Scott--Zhang-type operator by 
\begin{align}
\widehat J_\coarse:L^2(\widehat\Omega)\to \widehat\XX_\coarse,\quad\widehat v\mapsto \sum_{z\in\widehat\NN^\act_\coarse\setminus \partial\widehat\Omega^\act}\Big( \int_{\widehat\Omega} \widehat B_{\coarse,z}^*\, \widehat v\,dt\Big)\,\widehat B_{\coarse,z}.
\end{align}
A similar operator has already been defined and analyzed, e.g., in \cite[Section~7.1]{variational}.
Indeed, the only difference in the definition is the considered index set $\widehat\NN^\act_\coarse\setminus \partial\widehat\Omega^\act$ instead of $\widehat\NN^\act_\coarse$, which guarantees that $\widehat J_\coarse\widehat v$ vanishes on the boundary; see Lemma~\ref{lem:basis of X}.
Along the lines of \cite[Proposition~7.7]{variational}, one can thus prove the following result, where the projection property~\eqref{eq:projection} follows immediately from \eqref{eq:duality}.

\begin{lemma} \label{lem:Scott}
Let $\widehat\TT_\bullet\in\widehat\T$.
Then, $\widehat J_\bullet$ is a projection, i.e.,    
\begin{align}\label{eq:projection}
\widehat J_\bullet\widehat V_\bullet=\widehat V_\bullet\quad\text{for all }\widehat V_\bullet\in\XX_\bullet.
\end{align}
Moreover, $\widehat J_\bullet$ is locally $L^2$-stable, i.e., there exists $C_J>0$ such that for all $\widehat T\in \widehat\TT_\bullet$
\begin{align}\label{eq:L2stability}
\norm{\widehat J_\bullet\widehat v}{L^2(\widehat T)}\le C_J \norm{\widehat v}{L^2\big(\bigcup \{\supp(\widehat B_{\coarse,z}):(z\in\NN^\act_\coarse\setminus \partial\widehat\Omega^\act)\wedge|\supp(\widehat B_{\coarse,z})\cap\widehat T|>0)\}\big)}\quad\text{for all }\widehat v\in L^2(\widehat\Omega).
\end{align}
The constant $C_J$  depends only  on $d$ and $(p_1,\dots,p_d)$.\hfill$\square$
\end{lemma}

With Lemma~\ref{lem:Scott} at hand, we next  prove \eqref{S:proj} in the parameter domain.
Let $\widehat T\in\widehat\TT_\bullet, \widehat v\in H_0^1(\widehat\Omega)$, and $\widehat V_\bullet\in \widehat\XX_\bullet$ such that 
$\widehat v|_{\pi^{k_{\rm proj}}_\bullet(\widehat T)}=\widehat V_\bullet|_{\pi^{k_{\rm proj}}_\bullet(\widehat T)}$, where $k_{\rm proj}:= k_{\rm supp}$ with $k_{\rm supp}$ from Lemma~\ref{lem:local support}.
With Lemma~\ref{lem:local support}, the fact that $\supp(\widehat B_{\coarse,z}^*)=\supp(\widehat B_{\coarse,z})$, and the projection property \eqref{eq:projection} of $\widehat J_\bullet$, we conclude that 
\begin{align*}
(\widehat J_\bullet \widehat v)|_{\widehat T}
=\sum_{z\in\widehat\NN^\act_\coarse\setminus \partial\widehat\Omega^\act}\Big(\int_{\widehat\Omega} \widehat B_{\coarse,z}^*\widehat v\,dt\Big)\,\widehat B_{\coarse,z}|_{\widehat T}
=\sum_{\substack{z\in\widehat\NN^\act_\coarse\setminus \partial\widehat\Omega^\act\\\supp(\widehat B_{\coarse,z})\subseteq\pi_\coarse^{\rm proj}(\widehat T)}}\Big(\int_{\widehat\Omega} \widehat B_{\coarse,z}^*\widehat V_\coarse\,dt \Big) \,\widehat B_{\coarse,z}|_{\widehat T}
=\widehat V_\coarse|_{\widehat T}=\widehat v|_{\widehat T}.
\end{align*}

Next, we prove \eqref{S:app}. 
We note that for the modified projection operator from \cite{variational}, this property is already found, e.g., in \cite[Proposition~7.8]{variational}. 
Let $\widehat T\in\widehat\TT_\bullet$, $\widehat v\in H_0^1(\widehat\Omega)$, and  $\widehat V_\bullet\in \widehat \XX_\bullet$.
By  \eqref{eq:projection}--\eqref{eq:L2stability} and Lemma~\ref{lem:local support}, it holds that 
\begin{align*}
\norm{(1-\widehat J_\bullet)\widehat v}{L^2(\widehat T)} \stackrel{\eqref{eq:projection}}{=}\norm{(1-\widehat J_\coarse)(\widehat v-\widehat V_\bullet)}{L^2(\widehat T)}\stackrel{\eqref{eq:L2stability}}{\lesssim} \norm{  \widehat v-\widehat V_\bullet}{L^2(\pi_\bullet^{k_{\rm supp}}(\widehat T))}.
\end{align*}
To proceed, we distinguish between two cases, first, $\pi_\bullet^{2k_{\rm supp}}(\widehat T)\cap \partial\widehat\Omega=\emptyset$ and, second, $\pi_\bullet^{2k_{\rm supp}}(\widehat T)\cap \partial\widehat\Omega\neq\emptyset$, i.e., if $\widehat T$ is far away from the boundary or not.
Since the elements in the parameter domain are hyperrectangular, these cases are equivalent to $|\pi_\bullet^{2k_{\rm supp}}(\widehat T)\cap \partial\widehat\Omega|=0$ \new{and} $|\pi_\bullet^{2k_{\rm supp}}(\widehat T)\cap \partial\widehat\Omega|>0$, \new{respectively}, where $|\cdot|$ denotes the $(d-1)$-dimensional measure.

In the first case, we proceed as follows:
Nestedness \eqref{eq:YY nested} especially proves that  $1\in\widehat\YY_0\subseteq\widehat\YY_\bullet$.
Thus, there exists a representation $1=\sum_{z\in\widehat\NN_\bullet^\act}c_{z}\widehat B_{\coarse,z}$. 
Indeed, \cite[Proposition]{variational} even proves that $c_z=1$ for all $z\in\widehat\NN^\act_\coarse$, i.e., the B-splines $\set{\widehat B_{\coarse,z}}{z\in\widehat\NN^\act_\coarse}$ form a partition of unity.
With Lemma~\ref{lem:local support}, we see that $|\supp(\widehat\beta)\cap\pi_\bullet^{k_{\rm supp}}(\widehat T)|>0$ implies that $\supp(\widehat\beta)\subseteq\pi_\bullet^{2k_{\rm supp}}(\widehat T)$.
Therefore, the  restriction satisfies that 
\begin{align*}
1=\sum_{z\in\widehat\NN^\act_\coarse}\widehat B_{\coarse,z}|_{\pi_\bullet^{k_{\rm supp}}(\widehat T)}
=\sum_{\substack{z\in\widehat\NN^\act_\coarse\\|\supp(\widehat B_{\coarse,z})\cap \pi_\bullet^{k_{\rm supp}}(\widehat T)|>0}}\widehat B_{\coarse,z}|_{\pi_\bullet^{k_{\rm supp}}(\widehat T)}
=\sum_{\substack{z\in\widehat\NN^\act_\bullet\\ \supp(\widehat  B_{\coarse,z})\subseteq \pi_\bullet^{2k_{\rm supp}}(\widehat T)}}\widehat B_{\coarse,z}|_{\pi_\bullet^{k_{\rm supp}}(\widehat T)}.
\end{align*}
We define 
\begin{align*}
\widehat V_\bullet:=\widehat v_{\pi_\bullet^{k_{\rm supp}}(\widehat T)}\sum_{\substack{z\in\widehat\NN^\act_\bullet\\\supp(\widehat B_{\coarse,z})\subseteq \pi_\bullet^{2k_{\rm supp}}(\widehat T)}}\widehat B_{\coarse,z},\quad\text{where }\widehat v_{\pi_\bullet^{k_{\rm supp}}(\widehat T)}:=| \pi_\bullet^{k_{\rm supp}}(\widehat T)|^{-1}\int_{\pi_\bullet^{k_{\rm supp}}(\widehat T)}\widehat v\,dt.
\end{align*}
In the second case, we set $\widehat V_\bullet:=0$.
In the first case, we apply the Poincar\'e inequality, 
whereas we use the Friedrichs inequality 
in the second case. 
In either case, we obtain that $\widehat V_\bullet\in\widehat\XX_\bullet$, and \eqref{eq:shape regular parameter1}--\eqref{eq:shape regular parameter2} show that
\begin{align}\label{eq:PF}
\norm{\widehat v-\widehat V_\coarse}{L^2(\pi_\coarse ^{k_{\rm supp}}(\widehat T))}\lesssim\diam(\pi_\bullet^{2k_{\rm supp}}(\widehat T))\norm{\nabla \widehat v}{L^2(\pi_\bullet^{2k_{\rm supp}}(\widehat T))}\simeq |\widehat T|^{1/d} \norm{\nabla\widehat v}{L^2(\pi_\bullet^{2k_{\rm supp}}(\widehat T))}.
\end{align}
The hidden constants depend only on $\widehat\TT_0$, $(p_1,\dots,p_d)$, and  the shape of the patch $\pi_\bullet^{k_{\rm supp}}(\widehat T)$ \new{or} the shape of $\pi_\bullet^{2k_{\rm supp}}(\widehat T)$ and of $\pi_\bullet^{k_{\rm supp}}(\widehat T)\cap \partial\widehat \Omega$.
However, by \eqref{eq:shape regular parameter1}--\eqref{eq:shape regular parameter2}, the number of different patch shapes is bounded itself by a constant which again depends only on $d$ and  $(p_1,\dots,p_d)$.

Finally, we prove \eqref{S:grad}.
Let again $\widehat T\in\widehat\TT_\bullet$ and $\widehat v\in H_0^1(\widehat\Omega)$. 
For all $\widehat V_\bullet\in\widehat \XX_\bullet$ that are constant on $\widehat T$, the projection property \eqref{eq:projection} implies that 
\begin{eqnarray*}
\norm{\nabla \widehat J_\bullet\widehat v}{L^2(\widehat T)}\stackrel{\eqref{eq:projection}}{=}\norm{\nabla \widehat J_\bullet(\widehat v-\widehat V_\bullet)}{L^2(\widehat T)}&\stackrel{\eqref{eq:parameter invest}}{\lesssim}& |\widehat T|^{-1/d}\,\norm{\widehat J_\coarse (v-\widehat V_\coarse)}{L^2(\widehat T)}\\
&\stackrel{\eqref{eq:L2stability}}{\lesssim}&|\widehat T|^{-1/d}\,\norm{\widehat v-\widehat V_\bullet}{L^2(\pi_\bullet^{k_{\rm supp}}(\widehat T))}.
\end{eqnarray*}
Arguing as before and using \eqref{eq:PF}, we conclude the proof.

\subsection{Oscillation properties}\label{sec:oscproof}
There exists  $\C{inv}'>0$ such that the following property~\eqref{O:inverse} holds for all $\TT_\coarse\in\T$:
\begin{enumerate}
\renewcommand{\theenumi}{O\arabic{enumi}}
\bf\item\normalfont\label{O:inverse}
\textbf{Inverse estimate in dual norm.}
For all ${W}\in\mathcal{P}(\Omega)$, it holds that $|T|^{1/d} \norm{{W}}{L^2(T)}\le \C{inv}' \, \norm{{W}}{H^{-1}(T)}$. 
\end{enumerate}

Moreover, there exists $\C{lift}>0$  such that for all $\TT_\coarse\in\T$ and all $T,T'\in\TT_\coarse$ with non-trivial $(d-1)$-dimensional intersection $E:=T\cap T'$,  there exists a lifting   operator $ L_{\coarse,E}:\set{W|_E}{W\in\mathcal{P}(\Omega)} \to H_0^1(T\cup T')$ with the following properties~\eqref{O:dual}--\eqref{O:grad}:
\begin{enumerate}
\renewcommand{\theenumi}{O\arabic{enumi}}
\setcounter{enumi}{1}
\bf\item\normalfont
\textbf{Lifting inequality.}
For all ${W}\in\mathcal{P}(\Omega)$, it  holds that $\int_E{W}^2\,dx \le \C{lift} \int_E\,L_{\coarse,E} ({{W|_E}})W\,dx$.
\label{O:dual}
\bf\item\normalfont\label{O:stab}
\textbf{$\boldsymbol{L^2}$-control.}
For all  ${W} \in \mathcal{P}(\Omega)$, it holds that $\norm{L_{\coarse,E}({W|_E})}{L^2(T\cup T')}^2\le \C{lift} |T\cup T'|^{1/d}\, \norm{{W}}{L^2(E)}^2$.
\bf\item\normalfont\label{O:grad}
\textbf{$\boldsymbol{H^1}$-control.}
For all  ${W} \in \mathcal{P}(\Omega)$, it holds that $\norm{\nabla L_{\coarse,E}({W|_E})}{L^2(T\cup T')}^2\le  \C{lift} |T\cup T'|^{-1/d} \,\norm{{W}}{L^2(E)}^2$.
\end{enumerate}

The properties can be proved along the lines of \cite[Section~5.11--5.12]{igafem}, where they are proved for polynomials on hierarchical meshes; see also \cite[Section~4.5.11--4.5.12]{diss} for details.
The proofs rely on standard scaling arguments and the existence of a suitable bubble function.
The involved constants thus depend only on $d$, $\C{\gamma}$, and $( q_1,\dots, q_d)$.

\section{Possible Generalizations}\label{sec:generalizations}
In this section, we briefly discuss several easy generalizations of Theorem~\ref{thm:main}.
We note that all following generalizations are compatible with each other, i.e., Theorem~\ref{thm:main} holds analogously for rational T-splines in arbitrary dimension $d\ge2$ on 
geometries $\Omega$ that are initially non-uniformly meshed if one uses arbitrarily graded mesh-refinement.
If $d=2$, one can even employ rational T-splines of arbitrary degree $p_1,p_2\ge2$. 

\subsection{Rational T-splines}\label{subsec:rational}
Instead of the ansatz space $\XX_\bullet$, one can use rational hierarchical splines, i.e., 
\begin{align}
\XX_\bullet^{W_0}:=\Big\{\frac{V_\bullet}{W_0}:V_\bullet\in\XX_\bullet\Big\},
\end{align}
where $W_0\in\YY_0$ with $W_0>0$  is a fixed positive weight function.
In this case, the corresponding basis  consists of NURBS instead of B-splines.
Indeed, the mesh properties \eqref{M:shape}--\eqref{M:dual}, the refinement properties \eqref{R:sons}--\eqref{R:overlay}, and the oscillation properties \eqref{O:inverse}--\eqref{O:grad}  from Section~\ref{sec:proof} 
are independent of the discrete spaces.
To verify the validity of Theorem~\ref{thm:main} in the NURBS setting, it thus only remains to verify the properties \eqref{S:inverse}--\eqref{S:grad} for the NURBS finite element spaces.
The inverse estimate \eqref{S:inverse} follows similarly as in Section ~\ref{subsec:ansatz} since we only consider a fixed and thus uniformly bounded  weight function $W_0\in\YY_0$.
The properties \eqref{S:nestedness}--\eqref{S:local} depend only on the numerator of the NURBS functions and thus transfer.
To see \eqref{S:proj}--\eqref{S:grad}, one can proceed as in Section~\ref{subsec:ansatz}, where the corresponding Scott--Zhang-type operator $J_\bullet^{W_0}:L^2(\Omega)\to \XX_\bullet^{W_0}$ now  reads 
$J_\bullet^{W_0}v:={J_\bullet(vW_0)}/{W_0}$ {for all }$v\in L^2(\Omega)$.
Overall,  the involved constants then depend additionally on $W_0$.

\subsection{Non-uniform initial mesh}\label{subsec:non-uniform}
By definition, $\widehat\TT_0$ is a uniform tensor-mesh. 
Instead one can also allow for non-uniform  tensor-meshes 
\begin{align}\label{eq:initial mesh}
\widehat\TT_0=\Big\{\prod_{i=1}^d[a_{i,j},a_{i,j+1}]:i\in\{1,\dots,d\}\wedge j\in\{0,\dots,N_i-1\}\Big\}, 
\end{align}
where $(a_{i,j})_{j=0}^{N_i}$ is a strictly increasing vector with $a_{i,0}=0$ and $a_{i,N_i}=N_i$, and adapt the corresponding definitions accordingly.
In particular, for the refinement, the definition \eqref{eq:neighbors} of neighbors of an element has to be adapted and depends on $\widehat\TT_0$.
To circumvent this problem, one can transform the non-uniform mesh via some \new{function} $\varphi$ to a uniform one, perform the refinement there, and then transform the refined mesh back via $\varphi^{-1}$. 
Indeed, for each $i\in\{1,\dots,d\}$, there exists a continuous strictly monotonously increasing function $\varphi_i:[0,N_i]\to[0,N_i]$ that affinely maps any interval $[a_{i,j},a_{i,j+1}]$ to $[j,j+1]$. 
Then, the resulting tensor-product $\varphi:=\varphi_1\otimes\dots\otimes\varphi_d:\overline{\widehat\Omega}\to\overline{\widehat\Omega}$ defined as in~\eqref{eq:Bz} is a bijection. 
To prove  the mesh properties \eqref{M:shape}--\eqref{M:dual} and the refinement properties \eqref{R:sons}--\eqref{R:overlay}, one first verifies them on transformed meshes $\set{\varphi(\widehat T)}{\widehat T\in\widehat\TT_\coarse}$ as in Section~\ref{subsec:M true}--\ref{sec:refinement props}, and then transforms these results via $\gamma\circ\varphi^{-1}$ to physical meshes $\TT_\coarse$. 
The space properties \eqref{S:inverse}--\eqref{S:grad} and the oscillation properties \eqref{O:inverse}--\eqref{O:grad}  follow as in Section~\ref{subsec:ansatz}--\ref{sec:oscproof}. 
\new{All} involved constants depend additionally on $\widehat\TT_0$.


\subsection{Arbitrary grading} 
Instead of dividing the refined elements into two sons, one can also divide them into $m$ sons, where $m\ge2$ is a fixed integer.
Indeed, such a grading parameter $n$ has already been proposed and analyzed in \cite{morgensternT2} to obtain a more localized refinement strategy.
The proofs hold verbatim, but the constants depend additionally on $m$.

\subsection{Arbitrary dimension $d\ge2$}
\cite[Section~5.4~and~5.5]{diss_morgenstern} generalizes T-meshes, T-splines, and the refinement strategy developed in \cite{morgensternT2} for $d=3$ to arbitrary $d\ge2$. 
We note that the resulting refinement for $d=2$ does not coincide with the refinement from \cite{morgensternT1} that we consider in this work.
Instead, the latter leads to a smaller mesh closure. 
However, Theorem~\ref{thm:main} is still valid if the refinement strategy from \cite[Section~5.4~and~5.5]{diss_morgenstern} is used for $d\ge2$. 
Indeed, the mesh properties \eqref{M:shape}--\eqref{M:dual}  essentially follow from \eqref{eq:shape regular parameter1}--\eqref{eq:shape regular parameter2}, which \new{are} stated in \cite[Lemma~5.4.10]{diss_morgenstern}.
The properties \eqref{R:sons}--\eqref{R:reduction} are satisfied by definition, \eqref{R:closure} is proved in \cite[Section~5.4.2]{diss_morgenstern}, and \eqref{R:overlay} follows along the lines of \cite[Section~5]{morgensternT1}.
The space properties~\eqref{S:inverse} and \eqref{S:local}--\eqref{S:grad} can be verified as in Section~\ref{subsec:ansatz}, where the required dual-compatability is found in \cite[Theorem~5.3.14~and~5.4.11]{diss_morgenstern}. 
Nestedness~\eqref{S:nestedness} is proved in \cite[Theorem~5.4.12]{diss_morgenstern}.
The oscillation properties \eqref{O:inverse}--\eqref{O:grad} follow as in Section~\ref{sec:oscproof}.

\subsection{Arbitrary polynomial degrees $(p_1,\dots,p_d)$ for $d=2$}
In \cite{beirao}, T-splines of arbitrary degree have been analyzed for $d=2$. 
Depending on the degrees $p_1,p_2\ge2$, the corresponding basis functions are associated with elements, element edges, or, as in our case, with nodes.
We only restricted to odd degrees for the sake of readability. 
Indeed, the work \cite{morgensternT1} allows for arbitrary $p_1,p_2\ge2$.
In particular, all cited results of \cite{morgensternT1} are also valid in this case, and Theorem~\ref{thm:main} follows along the lines of Section~\ref{sec:proof}.
However, to the best of our knowledge, T-splines of arbitrary degree have not been investigated for $d>2$.

\section*{Acknowledgement} The authors acknowledge support through the Austrian Science Fund (FWF) under grant P29096 
and grant W1245.

\bibliographystyle{alpha}
\bibliography{literature}

\end{document}